\documentclass{amsart}
\usepackage[latin1]{inputenc}
\usepackage{amsmath}
\usepackage{amsthm}
\usepackage{amsfonts}
\usepackage{amssymb}
\usepackage[T1]{fontenc}
\usepackage{float}
\usepackage[all]{xy}
\usepackage{tikz}
\usepackage{calrsfs}
\usepackage{bbm}
\usepackage{color}
\usepackage{hyperref}

\theoremstyle{plain}
\newtheorem{theorem}{Theorem}[section]
\newtheorem{prop}[theorem]{Proposition}
\newtheorem{lem}[theorem]{Lemma}

\theoremstyle{definition}
\newtheorem{defi}[theorem]{Definition}

\newtheorem{exmp}[theorem]{Example}

\def\Gr{{\rm{Gr}}}

\def\dim{{\rm{dim}\,}}

\def\<{\left<}
\def\>{\right>}

\def\d{{\partial}}
\def\k{\mathbf{k}}
\def\kQ{\k Q}

\def\ens#1{\left\{ #1 \right\}}

\def\fl{{\longrightarrow}\,}

\def\CC{{\mathcal{C}}}

\def\A{{\mathbf{A}}}
\def\C{{\mathbf{C}}}
\def\T{{\mathbf{T}}}

\def\Q{{\mathbb{Q}}}
\def\Z{{\mathbb{Z}}}

\def\1{\mathbbm{1}}
\def\c{{\mathbf c}}
\def\x{{\mathbf x}}
\def\e{{\mathbf e}}
\def\f{{\mathbf f}}
\def\g{{\mathbf g}}
\def\v{{\mathbf v}}
\def\add{{\rm{add}}\,}

\def\ens#1{\left\{ #1 \right\}}
\def\Ext{{\rm{Ext}}}

\def\Hom{{\rm{Hom}}}

\def\Ob{{\rm{Ob}}}

\def\modd{{\textrm{mod-}\,}}
\def\reg{{\textrm{reg-}\,}}
\def\rr{{\textrm{reg}^0\textrm{-}\,}}

\def\1{\mathbbm{1}}

\def\Ax#1{\textbf{(A#1)}}
\def\SM{(S,M)}
\newcommand{\subsubsubsection}[1]{\emph{ #1}. }

\newcommand{\dd}{{\mathbf d}}

\title[Atomic bases in cluster algebras of types $A$ and $\widetilde A$]{Atomic bases in cluster algebras of types $A$ and $\widetilde A$} 
\author{Gr\'egoire Dupont and Hugh Thomas}
\address{Universit\'e de Sherbrooke, Sherbrooke QC, Canada}
\email{gregoire.dupont@usherbrooke.ca}
\address{University of New Brunswick, Fredericton NB, Canada}
\email{hthomas@unb.ca}
\date{\today}

\begin{document}

\begin{abstract}
	We give explicit atomic bases of arbitrary coefficient-free cluster algebras of types $A$ and $\widetilde A$.
\end{abstract}

\maketitle

\setcounter{tocdepth}{1}
\tableofcontents

\section{Introduction and main results}
	\subsection{Cluster algebras}
		Cluster algebras were introduced by Fomin and Zelevinsky in the early 2000's in order to provide a combinatorial framework for studying total positivity and dual canonical bases in semisimple groups \cite{cluster1}. Since then, cluster algebras have shown interactions with various areas of mathematics like combinatorics, Lie theory, Poisson geometry, Teichm\"uller theory, mathematical physics and representation theory.

		A (skew-symmetric) cluster algebra is defined from a \emph{seed}, that is a pair $(Q,\x)$ where $Q$ is a finite connected quiver with $n$ vertices and without oriented cycles of length $l \leq 2$ and where $\x = (x_1, \ldots, x_n)$ is a $n$-tuple of variables, called the \emph{cluster} of the seed. A combinatorial process, called \emph{mutation}, allows one to define recursively a (possibly infinite) family of seeds. The (coefficient-free) \emph{cluster algebra }$\mathcal A_Q$ is the $\Z$-subalgebra of the \emph{ambient field} $\Q(x_1, \ldots, x_n)$ generated by the union of all the clusters of the seeds arising from this mutation procedure. It is therefore naturally equipped with a $\Z$-module structure. A free generating set for this $\Z$-module structure is called a \emph{$\Z$-linear basis} of $\mathcal A_Q$.

		The cluster structure of $\mathcal A_Q$ naturally endows it with a distinguished set of elements, the \emph{cluster monomials}, which are the monomials in cluster variables all belonging to a single cluster. The set of cluster monomials in $\mathcal A_Q$ is denoted by $\mathcal M_Q$. As was proved for instance by Lampe for (quantum) cluster algebras of type $A$, see \cite{Lampe:typeA}, this set plays a prominent role in the construction of $\Z$-linear bases of $\mathcal A_Q$ which are of interest with respect to the study of dual canonical bases.

		A remarkable fact about cluster algebras is the so-called \emph{Laurent phenomenon}, proved in \cite{cluster1}, which asserts that for any cluster $\c = (c_1, \ldots, c_n)$ in $\mathcal A_Q$, the cluster algebra $\mathcal A_Q$ is a subring of $\Z[c_1^{\pm 1}, \ldots, c_n^{\pm 1}]$. An element in $\mathcal A_Q$ is called \emph{positive} if it belongs to the semiring $\Z_{\geq 0}[c_1^{\pm 1}, \ldots, c_n^{\pm 1}]$ for any cluster $\c = (c_1, \ldots, c_n)$ in $\mathcal A_Q$. We denote by $\mathcal A_Q^+$ the cone of positive elements in $\mathcal A_Q$. The \emph{positivity conjecture} asserts that every cluster monomial in $\mathcal A_Q$ is a positive element of $\mathcal A_Q$, see \cite{cluster1}. This conjecture was in particular established for cluster algebras with a bipartite seed \cite{Nakajima:cluster} and for cluster algebras arising from surfaces \cite{ST:unpunctured,MSW:positivity} but remains open in general. Note that this latter class of cluster algebras contains the class of cluster algebras of types $A$ and $\widetilde A$ which are considered in the present article.

		A $\Z$-basis $\mathcal B$ of $\mathcal A_Q$ is called an \emph{atomic basis} (or a \emph{canonically positive basis}) of $\mathcal A_Q$ if 
		$$\mathcal A_Q^+ = \bigoplus_{b \in \mathcal B} \Z_{\geq 0}b.$$
		The definition of atomic bases, which first appeared in \cite{shermanz}, was motivated by the positivity of the structure constants for multiplication of dual canonical bases elements. Note that it follows from the definition that if an atomic basis exists, then it is unique. Therefore, under the existence hypothesis, we can speak of \emph{the} atomic basis of $\mathcal A_Q$. However, the problem of showing the existence of this atomic basis of $\mathcal A_Q$ remains wide open in general. 

		If $\mathcal A_Q$ is of finite type in the sense of \cite{cluster2}, Cerulli recently proved that the atomic basis coincides with the set of cluster monomials of $\mathcal A_Q$, see \cite{Cerulli:finitetype}. If $\mathcal A_Q$ is not of finite type, it was observed in \cite{shermanz} that the set of cluster monomials does not necessarily generate the cluster algebra as a $\Z$-module, and therefore is not the atomic basis of $\mathcal A_Q$. In the particular cases where $Q$ is an affine quiver of type $\widetilde A_{1,1}$ or $\widetilde A_{2,1}$, the atomic bases were made explicit in \cite{shermanz} and \cite{Cerulli:A21} respectively. In this article, we generalise this construction to arbitrary quivers of affine type $\widetilde A$ and we provide a new, short and elementary proof of Cerulli's result for cluster algebras of type $A$.

		The article is organised as follows. In the remainder of this section we recall the necessary background on cluster algebras from surfaces and their connection to representation theory in order to state our main results. In Section \ref{section:typeA}, we prove that cluster monomials form the atomic basis in a cluster algebra of type $A$. In Section \ref{section:BQgeometric}, we give a combinatorial interpretation to a conjectural formula provided in \cite{Dupont:qChebyshev} for the atomic basis in a cluster algebra of type $\widetilde A$. Finally, we prove in Section \ref{section:BQatomic} that this conjectural atomic basis is indeed the atomic basis in type $\widetilde A$.

	\subsection{Marked surfaces}
		Following \cite{FST:surfaces}, we define an (unpunctured) \emph{marked surface} as a pair $\SM$ where $S$ is a connected oriented 2-dimensional Riemann surface with non-empty boundary $\d S$ and $M$ is a finite set of marked points on $\d S$ such that each connected component of $\d S$ contains at least one marked point. Moreover, we assume that $\SM$ is not homeomorphic to a disc with less than three marked points. 

		For any $n \geq 1$, we denote by $\Pi_n$ the marked surface consisting of a disc with $n+3$ marked points on the boundary, which we sometimes refer to as the \emph{$n+3$-gon}. For $p,q \geq 1$, we denote by $C_{p,q}$ the marked surface consisting of an annulus with $p$ marked points $\iota_1, \ldots, \iota_p$ on a boundary component $\iota$, called the \emph{inside} and $q$ marked points $o_1, \ldots, o_q$ on the other boundary component, called the \emph{outside}, see Figure \ref{fig:exmp} below. 

		\begin{figure}[H]
			\begin{center}
				\begin{tikzpicture}[scale = .5]

					\draw[fill=gray!20] (0,0) circle (3);
					\draw[thick] (0,0) circle (3);
			
					\foreach \x in {0,45,...,315}
					{
						\fill +(\x:3) circle (.1);
					}

					\draw[fill=gray!20] (8,0) circle (3);
					\draw[fill=white] (8,0) circle (1);

					\draw[thick] (8,0) circle (3);
					\draw[thick] (8,0) circle (1);
			
					\fill (8,3) circle (.1);
					\fill (8,1) circle (.1);
					\fill (8,-1) circle (.1);

					\fill (8,1) node [below] {\tiny $\iota_1$};
					\fill (8,-1) node [above] {\tiny $\iota_2$};

					\fill (8,3) node [above] {\tiny $o_1$};
				\end{tikzpicture}
			\end{center}
			\caption{The marked surfaces $\Pi_5$ and $C_{2,1}$.}\label{fig:exmp}
		\end{figure}
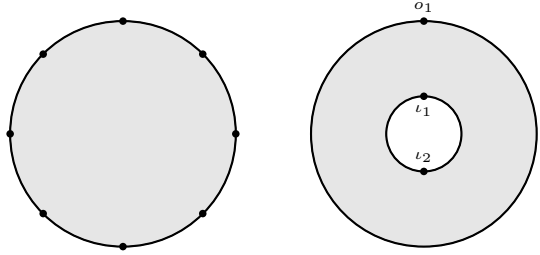

		Let $\SM$ be a marked surface. When we consider curves ``up to isotopy'' in $\SM$, we always mean ``up to isotopy with respect to the set $M$ of marked points''. Given two isotopy classes $\gamma$ and $\gamma'$ of curves in $\SM$, we define \emph{the number of intersections} $|\gamma \cap \gamma'|$ as the minimal number of intersections in the interior of $S$ of representatives of the two isotopy classes. Note that this number is reached if, once a hyperbolic structure is fixed on $\SM$, we consider geodesic representatives of $\gamma$ and $\gamma'$. Therefore, if one does not want to work up to isotopy, one can fix a hyperbolic structure and work with geodesic representatives. We say that two isotopy classes of curves are \emph{compatible} if they are the same or if their number of intersections is zero. An isotopy class $\gamma$ of curves in $\SM$ is called \emph{without self-intersection} if $|\gamma \cap \gamma| = 0$.

		A \emph{boundary segment} is a connected component of $\d S \setminus M$. We denote by $\C\SM$ the set of isotopy classes of curves joining two marked points, not isotopic to a boundary segment. An \emph{arc} in $\SM$ is an isotopy class of curves in $\SM$ joining two marked points, which is without self-intersection, and which is not isotopic to a boundary segment. We denote by $\A\SM$ the set of all arcs in $\SM$ and by $\T\SM$ the set of finite (possibly empty) families of pairwise compatible elements in $\A\SM$, considered with multiplicity. A curve (or its isotopy class) is called \emph{peripheral} if both its endpoints lie on a same boundary component and it is called \emph{bridging} otherwise. 

		A \emph{triangulation} $T$ of $\SM$ is a maximal set of pairwise distinct compatible arcs in $\SM$. Given a triangulation $T$ of $\SM$, one can associate to it a certain quiver $Q_T$ without loops and 2-cycles and thus a (coefficient-free) cluster algebra $\mathcal A_{Q_T}$, see \cite{FST:surfaces}. The cluster algebra constructed in this way is independent of the choice of the triangulation $T$. It is called \emph{the cluster algebra associated to the marked surface $\SM$} and is denoted by $\mathcal A_{\SM}$. 

		It is well-known that there is a bijection between the set $\A\SM$ of arcs in $\SM$ and the set of cluster variables in $\mathcal A_{\SM}$. Moreover, this bijection induces a bijection between the set of triangulations of $\SM$ and the set of clusters of $\mathcal A_{\SM}$, inducing a bijection between $\T\SM$ and the set of cluster monomials in $\mathcal A_{\SM}$. Moreover, in this context, mutations of clusters corresponds to flips of triangulations, see \cite{FST:surfaces}. Using these bijections, we will usually abuse notations and identify arcs in $\SM$ with cluster variables in $\mathcal A_{\SM}$, triangulations in $\SM$ with clusters in $\mathcal A_{\SM}$ and collections in $\T\SM$ with cluster monomials in $\mathcal A_{\SM}$.

		With these notations, the cluster algebra $\mathcal A_{\Pi_n}$ is of Dynkin type $A_n$ for any $n \geq 1$ and the cluster algebra $\mathcal A_{C_{p,q}}$ is of affine type $\widetilde A_{p,q}$ for any $p,q \geq 1$.

		Following \cite{ST:unpunctured}, for any cluster $T$ in $\mathcal A_{\SM}$, we can define an explicit map 
		$$x^T_?:\left\{\begin{array}{rcl}
			\A\SM & \fl & \mathcal A_{\SM} \\
			\gamma & \mapsto & x^T_\gamma
		\end{array}\right.$$
		which sends an arc $\gamma$ in $\A\SM$ to the $T$-expansion of the corresponding cluster variables in $\mathcal A_{\SM}$. This map does not depend on the choice of the triangulation so that we simply write $x_?$ for $x^T_?$ and $x_\gamma$ for $x^T_\gamma$. Given a family $\Gamma \in \T\SM$, we denote by $x_\Gamma$ the corresponding cluster monomial in $\mathcal A_{\SM}$; in other words $x_\Gamma = \prod_{\gamma \in \Gamma} x_\gamma$.

	\subsection{Atomic bases in type $A$}

		Our first main result is a short elementary combinatorial proof of the fact that cluster monomials form the atomic basis of a cluster algebra of type $A$.
		\begin{theorem}\label{theorem:typeA}
			Let $n \geq 1$, then 
			$$\ens{x_\Gamma \ | \ \Gamma \in \T(\Pi_n)}$$
			is the atomic basis of the cluster algebra $\mathcal A_{\Pi_n}$. 

			Equivalently, the atomic basis of a (coefficient-free) cluster algebra $\mathcal A$ of type $A$ is the set of cluster monomials in $\mathcal A$.
		\end{theorem}
		Note that this result was first obtained by Cerulli \cite{Cerulli:finitetype} using representations of quivers with potential in the sense of \cite{DWZ:potentials}.

	\subsection{Atomic bases in type $\widetilde A$}
		In the context of cluster algebras of type $\widetilde A_{p,q}$, we will consider an additional family of isotopy classes of curves in $C_{p,q}$, which we call \emph{loops}. A \emph{loop} in $C_{p,q}$ is the isotopy class of a non-contractible closed curve which lies in the interior of $C_{p,q}$. Note that for any $m \geq 1$, there is a unique loop $z_m$ in $C_{p,q}$ with $m-1$ self-intersections. We set $z=z_1$ and we denote by 
		$$\hat {\A}(C_{p,q}) = \A(C_{p,q}) \sqcup \ens{z_m \ | \ m \geq 1}$$
		the set of all arcs and loops in $C_{p,q}$ and by $\hat\T (C_{p,q})$ the set of finite (possibly empty) families of pairwise compatible arcs or loops in $\hat \A(C_{p,q})$, considered with multiplicity, containing at most one loop.
		
		In Section \ref{section:formula}, for any triangulation $T$ of $C_{p,q}$, we will extend the domain of definition of the maps $x^T_?$ to $\hat \A (C_{p,q})$ and prove that these new maps still do not depend on the choice of the triangulation $T$, thus defining a map $x_?$ on $\hat \A (C_{p,q})$. As before, in order to simplify notations, if $\Gamma$ is any collection in $\hat \T(C_{p,q})$, we adopt the following notations:
		$$x_\Gamma = \prod_{\gamma \in \Gamma} x_\gamma$$
		with the convention that $x_\emptyset = 1$.

		Our main result is an explicit realisation of the atomic basis in any (coefficient-free) cluster algebra of type $\widetilde A$:
		\begin{theorem}\label{theorem:BQatomic}
			Let $p,q \geq 1$, then 
			$$\ens{x_\Gamma \ | \ \Gamma \in \hat \T(C_{p,q})}$$
			is the atomic basis of the cluster algebra $\mathcal A_{C_{p,q}}$. 
		\end{theorem}

	\subsection{A representation-theoretic interpretation}
		An explicit expression for the atomic basis of a cluster algebra associated to an affine quiver was conjectured in \cite[Conjecture 7.10]{Dupont:qChebyshev} in terms of representation theory of algebras. Our third main result is the proof of this conjecture for cluster algebras of type $\widetilde A$. Before stating it precisely, we need to recall some background concerning the representation-theoretic approach to cluster algebras of type $\widetilde A$.

		Let $p,q \geq 1$ and let $Q$ be a quiver of type $\widetilde A_{p,q}$, that is, an orientation of the cyclic diagram with $n=p+q$ vertices with $p$ arrows going clockwise and $q$ arrows going counterclockwise. We fix an algebraically closed field $\k$. 

		Since $Q$ is an acyclic quiver of type $\widetilde A$, its path algebra $\kQ$ is a tame hereditary algebra so that the category mod-$\kQ$ of finitely generated right $\kQ$-modules is well-understood. We refer the reader to \cite{ringel:1099} or \cite{SS:volume2} for classical results on the representation theory of such algebras.

		A $\kQ$-module $M$ is called \emph{rigid} if $\Ext^1_{\kQ}(M,M) = 0$. A connected component of the Auslander-Reiten quiver of mod-$\kQ$ is called \emph{regular} if it contains neither a projective nor an injective $\kQ$-module. A $\kQ$-module is called \emph{regular} if all its indecomposable direct summands belong to regular components of Auslander-Reiten quiver of mod-$\kQ$. We denote by $\reg\kQ$ the set of regular $\kQ$-modules and by $\rr\kQ$ the set of regular rigid $\kQ$-modules. The regular components of the Auslander-Reiten quiver of mod-$\kQ$ form a $\mathbb P^1(\k)$-family of tubes. At most two tubes have a rank strictly larger than one and these have respective ranks $p$ and $q$. A tube with a rank equal to 1 is called \emph{homogeneous}; otherwise, it is called \emph{exceptional}. It is known that an indecomposable module in a tube is rigid if and only if its quasi-length is strictly smaller than the rank of the tube in which it is contained. In particular, homogeneous tubes do not contain any rigid modules.

		Let $D^b(\modd \kQ)$ denote the bounded derived category of mod-$\kQ$. It is a triangulated category with suspension functor $[1]$ and Auslander-Reiten translation $\tau$. The \emph{cluster category} $\CC_Q$ is the orbit category of the functor $F=\tau^{-1}[1]$ in $D^b(\modd \kQ)$. It is a triangulated 2-Calabi-Yau category \cite{K,BMRRT} and, up to isomorphisms, the set of indecomposable objects in $\CC_Q$ can be identified with the disjoint union of the set of indecomposable $\kQ$-modules and the set of shifts of indecomposable projective modules. Therefore, we view $\kQ$-modules as objects in $\CC_Q$.

		An object $T$ in $\CC_Q$ is called \emph{cluster-tilting} if for any object $X$ in $\CC_Q$, the equality $\Ext^1_{\CC_Q}(T,X)=0$ holds if and only if $X$ belongs to the additive category $\add(T)$. It is well-known that there is a bijection between the set of cluster-tilting objects in $\CC_Q$ and the set of clusters in $\mathcal A_Q$ \cite{BMRRT,CK2}. Thus, we will usually identify cluster-tilting objects in $\CC_Q$ with clusters in $\mathcal A_Q$ (and with triangulations of $C_{p,q}$). This bijection induces a bijection between the set of isomorphism classes of rigid objects in $\CC_Q$ (that is, objects $M$ such that $\Ext^1_{\CC_Q}(M,M)=0$) and the set of cluster monomials in $\mathcal A_Q$.

		This bijection can be made explicit by using the so-called cluster characters, first introduced in \cite{CC} and whose definition was generalised in \cite{CK2,Palu}. For any cluster-tilting object $T$ in $\CC_Q$ we denote by $X^T_?$ the \emph{cluster character} on $\CC_Q$ associated to $T$ with values in the ring of Laurent polynomials in the cluster $T$. If $(c_1, \ldots, c_{n})$ is the cluster in $\mathcal A_Q$ corresponding to the cluster-tilting object $T$, the cluster character is a map 
		$$X^T_?: \Ob(\CC_Q) \fl \Z[c_1^{\pm 1}, \ldots, c_{n}^{\pm 1}]$$
		which endows the cluster algebra $\mathcal A_Q$ with a structure of a Hall algebra on the cluster category $\CC_Q$ in the sense that for any objects $M,N$ in $\CC_Q$, the product $X^T_MX^T_N$ is a linear combination of $X^T_Y$ where $Y$ runs over the middle terms of triangles involving $M$ and $N$ \cite{CK1,Palu:multiplication}. In particular, if $\Ext^1_{\CC_Q}(M,N) \simeq \k$, then $X^T_MX^T_N = X^T_B + X^T_{B'}$ where $B$ and $B'$ are the unique objects in $\CC_Q$ such that there exist triangles $M \fl B \fl N \fl M[1]$ and $N \fl B' \fl M \fl N[1]$, see \cite{CK2,Palu}. We refer the reader to \cite{Palu} for the precise definition.

		Since $Q$ is an affine quiver, the cluster character $X^T_?$ takes its values in the cluster algebra $\mathcal A_Q$ and for any object $M$ in $\CC_Q$, if $T$ and $T'$ are two distinct cluster-tilting objects in $\CC_Q$, then $X^T_M = X^{T'}_M$, see \cite{Dupont:genericvariables}. We will thus omit the reference to the cluster-tilting object $T$ and simply denote by $X_M$ the corresponding element in the cluster algebra $\mathcal A_Q$.

		We denote by $X_\delta$ the so-called \emph{generic variable of dimension $\delta$} in $\mathcal A_Q$, which is given by the image of any quasi-simple module in a homogeneous tube of the Auslander-Reiten quiver of mod-$\kQ$, see \cite{Dupont:genericvariables}.

		For any $m \geq 1$, we denote by $F_m$ the \emph{$m$th normalised Chebyshev polynomial of the first kind} defined by
		$$F_0(z) = 2,\, F_1(z) = z \textrm{ and } F_{m+1}(z)= zF_m(z) - F_{m-1}(z), \textrm{ for any } m \geq 1.$$
		They are characterised by
		$$F_m(t+t^{-1}) = t^{m}+t^{-m}, \textrm{ for any } m \geq 0.$$

		In \cite[Conjecture 7.10]{Dupont:qChebyshev}, it is conjectured that the set 
		$$\mathcal B_Q = \mathcal M_Q \sqcup \ens{X_R F_m(X_\delta) \ | \ m \geq 1, R \in \rr\kQ}$$
		is the atomic basis of $\mathcal A_Q$. 
		
		Out third main result is the following theorem:
		\begin{theorem}\label{theorem:BQgeometric}
			Let $Q$ be a quiver of type $\widetilde A_{p,q}$, then 
			$$\mathcal B_Q = \ens{ x_\Gamma \ | \ \Gamma \in \hat{\T}(C_{p,q})}.$$
		\end{theorem}
		Thus, combined with Theorem \ref{theorem:BQatomic}, this proves \cite[Conjecture 7.10]{Dupont:qChebyshev}.

\section{Proof of Theorem \ref{theorem:typeA}}\label{section:typeA}
	In this section, $Q$ is a quiver of type $A_n$ with $n \geq 1$. Proving Theorem \ref{theorem:typeA} amounts to showing the following three points:
	\begin{enumerate}
		\item[\Ax 1] The cluster monomials form a $\mathbb Z$-linear basis of $\mathcal A_Q$.
		\item[\Ax 2] All cluster monomials are positive elements of $\mathcal A_Q$.
		\item[\Ax 3] Every positive element of $\mathcal A_Q$ can be written as a $\mathbb Z_{\geq 0}$-linear combination of cluster monomials.
	\end{enumerate}

	In this setting, \Ax 1 and \Ax 2 are already well-known. \Ax 1 follows from \cite{CK1}. \Ax 2 follows from the explicit positive combinatorial formulas of \cite{S:An}; positivity can also be shown directly by using the Ptolemy relations. The remaining step is therefore to prove \Ax 3. 

	\subsection{A combinatorial formula for $A_n$ cluster variables}\label{ss:combform} 
		$\Pi_n$ is the disc with $n+3$ marked points on the boundary. Recall that the cluster variables of an $A_n$ cluster algebra are in bijection with $\C(\Pi_{n})$. We begin by recalling a formula for expanding the cluster variable corresponding to a given curve $\gamma$ in terms of the cluster variables corresponding to a triangulation $T$ of $\Pi_n$. 

		We write $\mathcal W^T_\gamma$ for the set of walks joining the two endpoints of $\gamma$ satisfying the following: 
			\begin{enumerate}
				\item Each edge of the walk is either an arc of $T$ or a boundary segment.
				\item No edge of the walk is immediately followed by the same edge in the reverse direction.
				\item The walk is of odd length; we number the edges $\alpha_1,\dots,\alpha_{2m+1}$.
				\item Each even-numbered edge crosses $\gamma$.
				\item The arc $\gamma$ crosses the even-numbered edges of the path in the same order that they appear on the walk.
			\end{enumerate}
			Such a walk is called a \emph{coloured $\gamma$-walk on $T$}.

			For each $w\in \mathcal W^T_\gamma$, define a Laurent monomial
			$$p(w)=\frac{x_{\alpha_1}x_{\alpha_3}\dots x_{\alpha_{2m+1}}} {x_{\alpha_2}x_{\alpha_4}\dots x_{\alpha_{2m}}}.$$
			Then: 
			\begin{equation}\label{stform} 
				x_\gamma = \sum_{w\in \mathcal W^T_\gamma} p(w).
			\end{equation}
			This result first appeared in print in \cite{S:An}, but it had also been noticed by others previously. 

	\subsection{Technical lemmas}
		The following lemmas are at the heart of our argument. 

		\begin{lem}\label{lemma:Atechone}
			Fix a triangulation $T$ of the disc $\Pi_n$, and let $\gamma\in \A(\Pi_n)$ which is not in $T$. Then any term in the $T$-expansion of $x_\gamma$ has negative degree with respect to the cluster variables corresponding to arcs of $T$ which cross $\gamma$. 
		\end{lem}

		\begin{proof} 
			Consider $w$ a coloured $\gamma$-walk on $T$, and suppose that its length is $2m+1$. The corresponding term in the expansion of $x_\gamma$ has $m+1$ factors in the numerator, and $m$ factors in the denominator. All the factors in the denominator correspond to arcs in $\Pi_n$ which cross $\gamma$. Neither the first nor the last edge of $w$ contributes to the degree, proving the lemma.
		\end{proof}

		\begin{lem} \label{lemma:Atech}
			Fix a triangulation $T$ of the disc $\Pi_n$, and let $\gamma\in \A(\Pi_n)$, with $\gamma\not\in T$. Suppose that $\beta$ is an arc of $\Pi_n$ which is compatible with $\gamma$. Then each term in the $T$-expansion of $x_\beta$ has non-positive degree with respect to arcs of $T$ which cross $\gamma$.
		\end{lem}

		\begin{proof} 
			If $\beta$ is in $T$, then it cannot cross $\gamma$, so its degree is zero. If $\beta=\gamma$, we are done by the previous lemma. Otherwise, choose $w$ a coloured $\beta$-walk on $T$. Let $P$ denote the union of the triangles through which $\gamma$ passes. 

			Suppose first that $\beta$ and $\gamma$ do not share an endpoint. Consider the even-position edges of $w$ which lie in the interior of $P$. Since edges which cross $\beta$ must be encountered on $w$ in the same order as on $\beta$, these even-position edges must be a consecutive string of even positions, say $w_{2i}, w_{2i+2}\dots, w_{2j}$. Since $w_{2i-2}$ crosses $\beta$ but doesn't cross $\gamma$, it lies outside $P$, and so do all previous edges; similarly for $w_{2j+2}$ and all subsequent edges. It follows that the total degree of $p(w)$ with respect to the edges crossed by $\gamma$ would be positive only if all the odd-numbered edges $w_{2i-1},\dots,w_{2j+1}$ also cross $\gamma$. But it would follow that $w$ crosses $\gamma$ an odd number of times, and thus that so does $\beta$, contradicting the fact that
			$\beta$ and $\gamma$ are compatible.

			Suppose next that $\beta$ and $\gamma$ share an endpoint, which we assume is the starting point. Suppose that the even-position edges of $w$ which lie in $P$ are
			$w_2,\dots,w_{2j}$. As in the previous case, the corresponding term in the expansion of $x_\beta$ would be positive only if $w_1,\dots,w_{2j+1}$ all also cross $\gamma$. But $w_1$ is incident to the endpoint of $\gamma$, so it does not cross $\gamma$. 
		\end{proof}

	\subsection{Proof of \Ax 3}
		Let $y$ be a positive element in $\mathcal A_Q$. Write:
		$$y = \sum _{\Gamma \in \T(\Pi_n)} \lambda_\Gamma(y) x_\Gamma.$$

		Choose a particular $\Gamma$ appearing in the sum. We wish to show that $\lambda_\Gamma(y)$ is positive. Let $T$ be a triangulation of $\Pi_n$ which is compatible with $\Gamma$. 

		Consider some collection $\Sigma \in \T(\Pi_n)$, with $\Sigma \ne \Gamma$. We wish to show that the $T$-expansion of $x_{\Sigma}$ does not include any term $x_\Gamma$. 

		If $\Sigma$ consists of arcs from $T$, then $x_{\Sigma}$ is its own $T$-expansion, and we are done. Suppose otherwise. Let $\sigma$ be an arc in $\Sigma$ which is not an arc of $T$. Lemma \ref{lemma:Atechone} tells us that each term in the $T$-expansion of $x_\sigma$ is of negative degree with respect to the edges which cross $\sigma$. At the same time, Lemma \ref{lemma:Atech} tells us that each term in the $T$-expansion of the other factors of $x_{\Sigma}$ are of non-positive degree with respect to the same grading. It follows that each term in the $T$-expansion of $x_\Sigma$ is of negative degree with respect to this grading, which implies in particular that it contains no term $x_\Gamma$.

		Therefore, $\lambda_\Gamma(y)$ equals the coefficient of $x_\Gamma$ in the $T$-expansion of $y$, which is therefore non-negative, as desired. \hfill \qed

\section{Proof of Theorem \ref{theorem:BQgeometric}}\label{section:BQgeometric}

	\subsection{A combinatorial formula for curves}
		We begin by recalling the extension of the combinatorial formula which we stated in Section \ref{ss:combform}, to the case of a general marked surface $(S,M)$, following \cite{ST:unpunctured}. 

		There, a formulation is given of the rules (1)--(5) for coloured $\gamma$-walks on a triangulation $T$. However, in practice, it is easier to use the following reformulation, which is immediate from Lemma 4.7 of \cite{ST:unpunctured}. 

		Take $\gamma$, and lift it to an arc $\widetilde \gamma$ in the universal cover. Lift the triangulation $T$ to a triangulation $\widetilde T$. Then define the coloured $\gamma$-walks on $T$ to be the images on $S$ of the 
coloured $\widetilde \gamma$-walks on $\widetilde T$. The main theorem of \cite{ST:unpunctured} is that \eqref{stform} still holds in this case, i.e., $x_\gamma$ is the  sum of the terms $p(w)$, as $w$ runs through the coloured $\gamma$-walks on $T$.  

		We also want to assign an element of the cluster algebra to a curve which runs between two marked points and which has self-intersections. For such a curve $\gamma$, we take the same definition of $\mathcal W^T_\gamma$ as above, and define $x^T_\gamma$ to be the sum of $p(w)$ over all coloured $\gamma$-walks $w$. Note that, {\it a priori}, this definition is not independent of the choice of $T$. 

		\begin{lem}\label{lem:mutxtC}
			Let $T$ and $T'$ be two triangulations of $\SM$. Then for any curve $\gamma \in \C\SM$, we have $x^T_\gamma = x^{T'}_\gamma$.
		\end{lem}
		\begin{proof}
			Fix $\hat \gamma$, a lift of $\gamma$ to the universal cover of $\SM$. Let $\hat T$ and $\hat T'$ be lifts of the triangulations $T$ and $T'$ to the universal cover of $\SM$. Let $P$ be the union of the triangles of $\hat T$ which intersect $\hat \gamma$. Consider the expansion of $x_{\hat \gamma}$ in terms of cluster variables corresponding to edges of $\hat T$. Only edges from $P$ are involved in the expansion, and the expansion is independent of the triangulation outside $P$. 

			Now, if we take $D$ to be the union of the triangles of $T'$ which intersect $P$, we have a disc with marked points, which can be triangulated either by the restriction of $T'$, or by extending the triangulation of $P$ by $T$. In the type $A$ cluster algebra associated to $D$, we know that the expansions of $x_{\hat \gamma}$ with respect to these two triangulations agree. 

			When we pass down to $(S,M)$, these two expansions yield the expansions of $x_\gamma$ with respect to $T$ and $T'$, so these also coincide. 
		\end{proof}

		From now on, we will in general omit the reference to the triangulation and for any curve $\gamma$ in $\SM$, the notation $x_\gamma$ will designate the element in the ambient field corresponding to $x^T_\gamma$ for any choice of triangulation $T$. The notation $x^T_\gamma$ will be kept only in order to specify the the explicit Laurent expansion in the cluster $T$.

	\subsection{Curves in $C_{p,q}$ and objects in $\CC_Q$}
		From now on, we assume that $\SM = C_{p,q}$ for some $p,q \geq 1$ and that $Q$ is an affine quiver of type $\widetilde A_{p,q}$. It is known that there is a bijection $\gamma \mapsto M_\gamma$ from the set $\C(C_{p,q})$ to the set of isomorphism classes of indecomposable objects in $\CC_Q$ which are not contained in an homogeneous tube \cite{BZ:clustercatsurfaces}. Let us make this bijection explicit for objects in exceptional tubes. 

		If $p=q=1$, there are no exceptional tubes in $\Gamma(\CC_Q)$. Thus, without loss of generality, we assume that $p >1$. We consider the set $\C^\iota(C_{p,q})$ of curves in $\C(C_{p,q})$ both of whose endpoints are the inside boundary $\iota$ of $C_{p,q}$ containing $p$ points. We denote by $m_i$ with $i \in \Z/p\Z$ the marked points on $\iota$. The orientation of $C_{p,q}$ induces an orientation of $\iota$ and following this orientation we can assume that the successor of $m_i$ is $m_{i+1}$ for any $i \in \Z/p\Z$. We denote by $\gamma_{m_i}^{(0)}$ the oriented arc with starting point $m_i$ and endpoint $m_{i+1}$ for any $i \in \Z/p\Z$. Finally, for any $l \geq 1$, we set $\gamma_{m_i}^{(l)} = \gamma_{m_{i+l}}^{(0)} \circ \gamma_{m_i}^{(l-1)}$. Figure \ref{fig:Cp} depicts the situation.

		\begin{figure}[H]
			\begin{center}
				\begin{tikzpicture}
					\draw (0,0) -- (8,0);
					\foreach \x in {0,2,...,8}
						\fill (\x,0) circle (.05);
					\fill (0,0) node [below] {$m_{i}$}; 
					\fill (2,0) node [below] {$m_{i+1}$}; 
					\fill (4,0) node [below] {$m_{i+2}$}; 
					\fill (6,0) node [below] {$m_{i+3}$}; 
					\fill (8,0) node [below] {$m_{i+4}$}; 

					\fill (1,0) node [below] {$\gamma_{m_i}^{(0)}$}; 
					\fill (3,0) node [below] {$\gamma_{m_{i+1}}^{(0)}$}; 
					\fill (5,0) node [below] {$\gamma_{m_{i+2}}^{(0)}$}; 
					\fill (7,0) node [below] {$\gamma_{m_{i+3}}^{(0)}$}; 
		
					\draw (0,0) arc (180:0:2); 
					\draw (0,0) arc (180:0:3); 
					\draw (0,0) arc (180:0:4);

					\fill (2,1.5) node {$\gamma_{m_i}^{(1)}$}; 
					\fill (3,2.5) node {$\gamma_{m_i}^{(2)}$}; 
					\fill (4,3.5) node {$\gamma_{m_i}^{(3)}$}; 
				\end{tikzpicture}
			\end{center}
			\caption{Arcs in $\C^\iota(C_{p,q})$}\label{fig:Cp}
		\end{figure}

		In $\Gamma(\CC_Q)$ there is an exceptional tube $\mathcal T_p$ of rank $p$. We denote by $R_i$, with $i \in \Z/p\Z$ the quasi-simple objects in $\mathcal T_p$. For any $l \geq 1$, we denote by $R_i^{(l)}$ the unique indecomposable object in $\mathcal T_p$ with quasi-socle $R_i$ and quasi-length $l$. Then the above bijection is given by $M_{\gamma_{m_i}^{(l)}} = R_i^{(l)}$ for any $i \in \Z/p\Z$ and any $l \geq 1$. We adopt the convention that $R_i^{(0)}=0$ for any $i \in \Z/p\Z$. The situation is depicted in Figure \ref{fig:tube}.

		\begin{figure}[H]
			\begin{center}
				\begin{tikzpicture}[scale = .5]
					\tikzstyle{every node}=[font=\tiny]
					\foreach \y in {0,2,...,4}
					{
						\foreach \x in {0,2,...,10}
						{
							\draw[->] (\x+1,\y+1) -- (\x+2,\y);
							\draw[->] (\x,\y) -- (\x+1,\y+1);
							\draw[->] (\x,\y+2) -- (\x+1,\y+1);
							\draw[->] (\x+1,\y+1) -- (\x+2,\y+2);
						}
					}
					\foreach \y in {6}
					{
						\foreach \x in {0,2,...,10}
						{
							\draw[->,dashed] (\x+1,\y+1) -- (\x+2,\y);
							\draw[->,dashed] (\x,\y) -- (\x+1,\y+1);
							\draw[->,dashed] (\x,\y+2) -- (\x+1,\y+1);
							\draw[->,dashed] (\x+1,\y+1) -- (\x+2,\y+2);
						}
					}
					\draw (0,0) -- (0,6);
					\draw (12,0) -- (12,6);
					\draw[dashed] (0,6) -- (0,8);
					\draw[dashed] (12,6) -- (12,8);

					\draw (0,0) circle (.075);
					\fill (0,0) node[below] {$R_0=R_0^{(1)}$};
					\draw (1,1) circle (.075);
					\fill (1.1,1.1) node[above] {$R_0^{(2)}$};
					\draw (2,2) circle (.075);
					\fill (2.1,2.1) node[above] {$R_0^{(3)}$};
					\draw (5,5) circle (.075);
					\fill (5.1,5.1) node[above] {$R_0^{(p)}$};
					\draw (8,8) circle (.075);
					\fill (8.1,8.1) node[above] {$R_0^{(l)}$};

					\draw (2,0) circle (.075);
					\fill (2,0) node[below] {$R_1$};

					\draw (10,0) circle (.075);
					\fill (10,0) node[below] {$R_{p-1}$};

					\draw (12,0) circle (.075);
					\fill (12,0) node[below] {$R_{0}$};

				\end{tikzpicture}
			\end{center}
			\caption{The tube $\mathcal T_p$}\label{fig:tube}
		\end{figure}
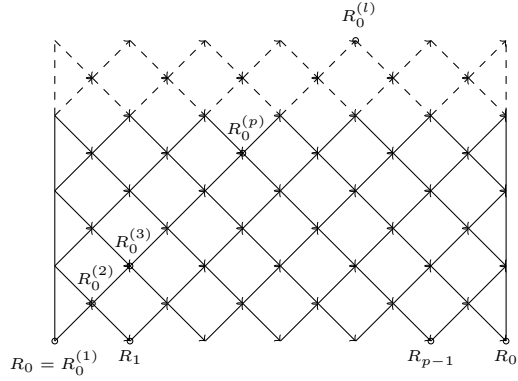
		If $q \geq 1$, one can write down a similar bijection between the set $\C^o(C_{p,q})$ of arcs whose both endpoints lie on the outside component $o$ containing $q$ points and the corresponding tube $\mathcal T_q$ of rank $q$ in $\Gamma(\CC_Q)$.

		We now compare the formula $x_?$ with the cluster character on the category $\CC_Q$.
		\begin{lem}\label{lem:xTXTCp}
			\begin{enumerate}
				\item Assume that $p >1$. Then $x_{\gamma} = X_{M_\gamma}$ for any $\gamma \in \C^\iota(C_{p,q})$.
				\item Assume that $q >1$. Then $x_{\gamma} = X_{M_\gamma}$ for any $\gamma \in \C^o(C_{p,q})$.
			\end{enumerate}
		\end{lem}
		\begin{proof}
			By symmetry it is enough to prove the first point. Fix a curve in $\C^\iota(C_{p,q})$. It is of the form $\gamma_{m_i}^{(l)}$ for some $i \in \Z/p\Z$ and some $l \geq 1$. We prove the result by induction on $l$. If $l =1$ then $\gamma_{m_i}^{(1)}$ has no self-intersection so that $\gamma_{m_i}^{(1)} \in \A(C_{p,q})$. In particular, $x_{\gamma_{m_i}^{(1)}}$ and $X_{R_i^{(1)}}$ are cluster variables. In order to prove that they are the same, it is enough to consider their expressions as Laurent polynomials with respect to a fixed cluster/triangulation. We fix the following triangulation $T$ in $C_{p,q}$:
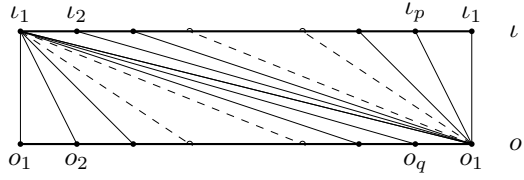
			\begin{figure}[H]
				\begin{center}
					\begin{tikzpicture}[scale = .75]
						\draw[thick] (0,0) -- (8,0);
						\fill (8.5,2) node [right] {$\iota$};
						\draw[thick] (0,2) -- (8,2);
						\fill (8.5,0) node [right] {$o$};
							
						\fill (0,2) node [above] {$\iota_1$};
						\fill (1,2) node [above] {$\iota_2$};
						\fill (7,2) node [above] {$\iota_p$};
						\fill (8,2) node [above] {$\iota_1$};

						\fill (0,0) node [below] {$o_1$};
						\fill (1,0) node [below] {$o_2$};
						\fill (7,0) node [below] {$o_q$};
						\fill (8,0) node [below] {$o_1$};

						\foreach \x in {0,1,2,6,7,8}
						{
							\fill (\x,0) circle (.05);
							\fill (\x,2) circle (.05);
						}
						\foreach \x in {3,5}
						{
							\draw[very thin,dashed] (\x,0) circle (.05);
							\draw[very thin,dashed] (\x,2) circle (.05);
						}

						\foreach \x in {0,1,2,6,7,8}
						{
							\draw[very thin] (0,2) -- (\x,0);
							\draw[very thin] (8,0) -- (\x,2);
						}
						\foreach \x in {3,5}
						{
							\draw[very thin,dashed] (0,2) -- (\x,0);
							\draw[very thin,dashed] (8,0) -- (\x,2);
						}
					\end{tikzpicture}
				\end{center}
			\end{figure}
			Then it is known that the denominator vector of $x^T_{\gamma_{m_i}^{(1)}}$ is given by the number of intersections of $\gamma_{m_i}^{(1)}$ with the arcs of $T$ \cite[Theorem 8.6]{FST:surfaces} and that the denominator vector of $ X^T_{R_i^{(1)}}$ is the dimension vector of the $\kQ$-module $R_i$ \cite[Theorem 3]{CK2}. Then, up to a cyclic permutation of the indices, the denominator vectors of $x^T_{\gamma_{m_i}^{(1)}}$ and $ X^T_{R_i^{(1)}}$ coincide. Thus, the two cluster variables are the same.

			Now assume that the result holds for $l \geq 1$. Using a standard covering argument and Ptolemy relations in the covering, it is easily verified that 
			$$x^T_{\gamma_{m_i}^{(l)}}x^T_{\gamma_{m_{i+1}}^{(l)}} = x^T_{\gamma_{m_i}^{(l+1)}}x^T_{\gamma_{m_{i+1}}^{(l-1)}}+1$$
			for any $i \in \Z/p\Z$.
			But on the other hand, it follows from \cite[Theorem 1]{Palu:multiplication} that 
			$$X^T_{R_i^{(l)}}X^T_{R_{i+1}^{(l)}} = X^T_{R_{i}^{(l+1)}}X^T_{R_{i+1}^{(l-1)}}+1$$
			for any $i \in \Z/p\Z$.
			Thus, $$x^T_{\gamma_{m_i}^{(l+1)}} = X^T_{R_{i}^{(l+1)}},$$
			which proves the induction step.
		\end{proof}

	\subsection{A formula for loops}\label{section:formula}
		Let $m \geq 1$ be an integer and let $T$ be a triangulation of $C_{p,q}$. Consider the annulus $C_{mp,mq}$ which is the $m$-fold cover of $C_{p,q}$. The triangulation $T$ induces naturally a triangulation of $C_{mp,mq}$, which is denoted by $\widetilde T$. Let $[0,1] \fl C_{mp,mq}$ denote a parametrisation of the meridian $\tilde z$ in $C_{mp,mq}$. The order in which $\tilde z$ intersects the bridging arcs of the triangulation $T$ with respect to this parametrisation induces an order on the bridging arcs of $\widetilde T$.

		\begin{defi}
			A \emph{coloured $m$-walk on $T$} is a walk of even length along the edges of the triangulation $\widetilde T$ and along the boundary components of $C_{mp,mq}$, whose edges are decorated with alternating $+$ and $-$ signs and such that:
			\begin{itemize}
				\item [\textbf{(P1)}] The walk is homotopic to $\tilde z$.
				\item [\textbf{(P2)}] Every edge decorated with a $-$ is a bridging arc of $T$.
				\item [\textbf{(P3)}] The walk goes forward in the sense that if two bridging arcs $\alpha$ and $\beta$ appear in the walk in this order, then $\alpha$ strictly precedes $\beta$ in the order induced by the parametrisation of $\tilde z$.
			\end{itemize}
			We denote by $\mathcal W^T_{z_m}$ the set of coloured $m$-walks on the triangulation $T$ in $C_{p,q}$.
		\end{defi} 

		\begin{defi}
			For any triangulation $T$ of $C_{p,q}$ and any $m \geq 1$, the Laurent polynomial $x^T_{z_m}$ in the cluster $T$ is the sum over all coloured $m$-walks $p$ on $T$ satisfying \textbf{(P1)}--\textbf{(P3)}, of the product $x(w)$ of the cluster variables in $T$ corresponding to $+$ edges in $w$, divided by the product of the cluster variables in $T$ corresponding to $-$ edges in $w$, with the convention that boundary arcs contribute as 1. In other words,
			$$x^T_{z_m} = \sum_{w \in \mathcal W^T_{z_m}} x(w).$$
		\end{defi}

		\begin{exmp}\label{exmp:Kronecker}
			We consider the triangulation $T$ of $C_{1,1}$ depicted in Figure \ref{fig:z1}. The coloured 1-walks on $T$ are depicted in Figure \ref{fig:z1} where edges with a negative colour appear in blue and edges with a positive colour appear in red.
			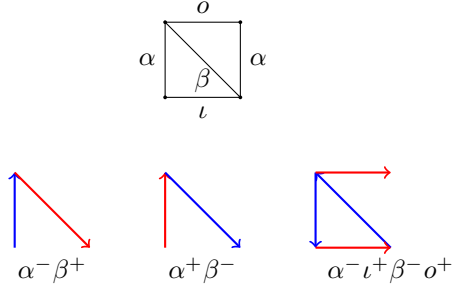
\begin{figure}[H]
				\begin{center}
					\begin{tikzpicture}[scale=.5]
						\draw (4,0) -- (6,0);
						\draw (4,2) -- (6,2);
						\foreach \x in {4,6}
						{
							\fill (\x,0) circle (.05);
							\fill (\x,2) circle (.05);
							
							\draw (\x,2) -- (\x,0);
						}
						\fill (4,1) node [left] {$\alpha$};
						\fill (6,1) node [right] {$\alpha$};

						\foreach \x in {4}
						{
							\draw (\x,2) -- (\x+2,0);
							\fill (\x+1,1) node [below] {$\beta$};
							\fill (1+\x,0) node [below] {$\iota$};
							\fill (1+\x,2) node [above] {$o$};
						}
						
						\draw[thick,blue,->] (0,-4) -- (0,-2);
						\draw[thick,red,->] (0,-2) -- (2,-4);
						\fill (1,-4) node [below] {$\alpha^- \beta^+$};

						\draw[thick,red,->] (4,-4) -- (4,-2);
						\draw[thick,blue,->] (4,-2) -- (6,-4);
						\fill (5,-4) node [below] {$\alpha^+\beta^- $};

						\draw[thick,blue,->] (8,-2) -- (8,-4);
						\draw[thick,blue,->] (10,-4) -- (8,-2);
						\draw[thick,red,->] (8,-2) -- (10,-2);
						\draw[thick,red,->] (8,-4) -- (10,-4);
						\fill (10,-4) node [below] {$\alpha^-\iota^+ \beta^- o^+$};
					\end{tikzpicture}
				\end{center}
				\caption{Coloured 1-walks on $\ens{\alpha,\beta}$ in type $\widetilde A_{1,1}$.}\label{fig:z1}
			\end{figure}
			
			Thus, we have 
			$$x^T_{z} = \frac{x_{\beta}}{x_{\alpha}} + \frac{x_{\alpha}}{x_{\beta}} + \frac{1}{x_\alpha x_\beta} = \frac{1+x_\alpha^2 + x_\beta^2}{x_\alpha x_\beta}.$$

			For $m=2$, the two-fold covering of $C_{1,1}$ is
			
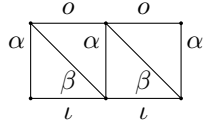
\begin{center}
				\begin{tikzpicture}[scale=.5]
					\draw (3,0) -- (7,0);
					\draw (3,2) -- (7,2);
					\foreach \x in {3,5,7}
					{
						\fill (\x,0) circle (.05);
						\fill (\x,2) circle (.05);
						
						\draw (\x,2) -- (\x,0);
					}
					\fill (3.1,1.5) node [left] {$\alpha$};
					\fill (5.1,1.5) node [left] {$\alpha$};
					\fill (6.9,1.5) node [right] {$\alpha$};

					\foreach \x in {3,5}
					{
						\draw (\x,2) -- (\x+2,0);
						\fill (\x+1,1) node [below] {$\beta$};
						\fill (1+\x,0) node [below] {$\iota$};
						\fill (1+\x,2) node [above] {$o$};
					}
				\end{tikzpicture}
			\end{center}
			
			Thus, the set of coloured 2-walks on $T$ consists of the following elements:
			\begin{itemize}
				\item $\alpha^-\iota^+\alpha^-o^+$
				\item $\alpha^-o^+\alpha^-\iota^+$
				\item $\beta^-\iota^+\beta^-o^+$
				\item $\beta^-o^+\beta^-\iota^+$
				\item $\alpha^-\beta^+\alpha^-\beta^+$
				\item $\beta^-\alpha^+\beta^-\alpha^+$
				\item $\iota^+ \beta^- o^+ \alpha^-\iota^+ \beta^- o^+ \alpha^-$.
			\end{itemize}
			Thus, we have 
			$$x^T_{z_2} = \frac{2}{x_\alpha^2} + \frac{2}{x_\beta^2} + \frac{x_\beta^2}{x_\alpha^2} + \frac{x_\alpha^2}{x_\beta^2}+ \frac{1}{x_\alpha^2 x_\beta^2} = (x^T_z)^2 -2 = F_2(x^T_z).$$
		\end{exmp}

	\subsection{Invariance under cluster change}

		\begin{lem}\label{lem:mutXtz}
			Let $T$ and $T'$ be two triangulations of $C_{p,q}$. Then for any $\gamma$ in $\hat \A(C_{p,q})$, we have $x^T_{\gamma} = x^{T'}_\gamma$.
		\end{lem}
		\begin{proof}
			For $\gamma$ an arc (not a loop), as we have already discussed, $x^T_\gamma$ and $x^{T'}_\gamma$ are two different expansions of the same cluster variable, so they are equal.  Thus, we only need to prove that $x^T_{z_m} = x^{T'}_{z_m}$ for any $m \geq 1$. Since any two triangulations are related by a sequence of mutations, it is enough to prove it in the case when $T$ and $T'$ are related by a single mutation.

			Suppose first that the mutation replaces a peripheral arc by another peripheral arc. This means that all the arcs of the quadrilateral in which the mutation takes place are peripheral. It follows that neither of the edges being mutated can appear in the formulas for $x^T_{z_m}$ and $x^{T'}_{z_m}$, so in this case we are done.

			Next suppose that the mutation replaces a peripheral arc $\sigma$ by a bridging arc $\tau$, and suppose the other four arcs are labelled $\alpha, \beta, \gamma, \delta$ as shown in Figure \ref{fig:mutation1} below.
			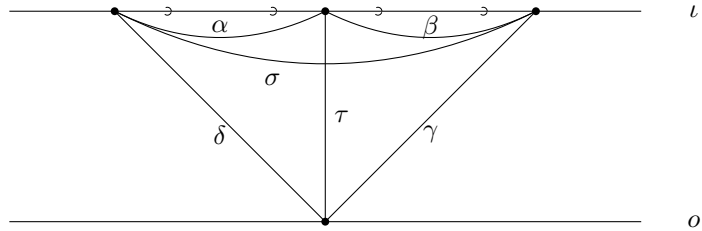
\begin{figure}[H]
				\begin{center}
					\begin{tikzpicture}[scale = .7]
						\draw (0,0) -- (12,0);
						\draw (0,4) -- (12,4);
						\fill (13,0) node {$o$};
						\fill (13,4) node {$\iota$};

						\fill (2,4) circle (.075);
						\draw[very thin, dashed] (3,4) circle (.075);
						\draw[very thin, dashed] (5,4) circle (.075);
						\fill (6,4) circle (.075);
						\draw[very thin, dashed] (7,4) circle (.075);
						\draw[very thin, dashed] (9,4) circle (.075);
						\fill (10,4) circle (.075);

						\fill (6,0) circle (.075);

						\draw (2,4) -- (6,0);
						\draw (6,4) -- (6,0);
						\draw (10,4) -- (6,0);

						\draw plot[domain=-2:2] ({\x+4},{3.5+2*(\x)^2/(4^2)});
						\draw plot[domain=-2:2] ({\x+8},{3.5+2*(\x)^2/(4^2)});
						\draw plot[domain=-4:4] ({\x+6},{3+(\x)^2/(4^2)});

						\fill (4,2) node [left, below] {$\delta$};
						\fill (8,2) node [right, below] {$\gamma$};
						\fill (4,4) node [below] {$\alpha$};
						\fill (8,4.1) node [below] {$\beta$};
						\fill (5,3) node [below] {$\sigma$};
						\fill (6,2) node [right] {$\tau$};
					\end{tikzpicture}
				\end{center}
				\caption{Replacing a peripheral arc by a bridging arc}\label{fig:mutation1}
			\end{figure}

			Suppose also that the coloured $m$-walk passes locally from left to right in the diagram. Since $\sigma$ is peripheral, it is decorated with a $+$.

			Each coloured $m$-walk on $T$ which includes $\sigma$ induces two coloured $m$-walks on $T'$ where $\sigma$ is replaced respectively by $\delta \tau \beta$ and by $\alpha \tau \gamma$. (If these replacements result in some edge being used twice consecutively in opposite directions, we cancel them out.) We claim that any coloured $m$-walk on $T'$ is obtained from a coloured $m$-walk on $T$ in this way. Suppose we start with a coloured $m$-walk on $T'$ containing $\tau$. It can only occur with a $-$ decoration, since on one end it is not adjacent to any bridging edges. Each coloured $m$-walk $w$ on $T'$ which includes $\tau$ from $\iota$ to $o$ must be of the form $\cdots \alpha^+ \tau^- \cdots$. Replacing $\alpha^+ \tau^-$ by $\sigma^+ \gamma^-$ in this expression, we obtain the desired coloured $m$-walk on $T$ inducing $w$. If $w$ uses $\tau$ from $o$ to $\iota$, then it must be of the form $\cdots \tau^- \beta^+ \cdots$. Replacing $\tau^- \beta^+$ by $\delta^- \sigma^+$ provides the desired coloured $m$-walk on $T$. This proves the claim.

			Finally, suppose that the mutation replaces one bridging $\sigma$ in $T$ arc by another bridging arc $\tau$ in $T'$, and suppose the other four arcs are labelled $\alpha, \beta, \gamma, \delta$ as shown in Figure \ref{fig:mutation2} below.
			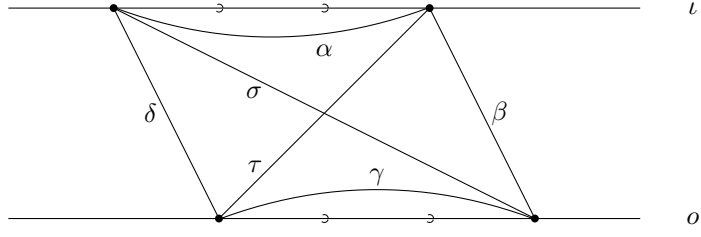
\begin{figure}[H]
				\begin{center}
					\begin{tikzpicture}[scale = .7]
						\draw (0,0) -- (12,0);
						\draw (0,4) -- (12,4);
						\fill (13,0) node {$o$};
						\fill (13,4) node {$\iota$};

						\fill (2,4) circle (.075);
						\draw[very thin, dashed] (4,4) circle (.075);
						\draw[very thin, dashed] (6,4) circle (.075);
						\fill (8,4) circle (.075);

						\fill (4,0) circle (.075);
						\draw[very thin, dashed] (6,0) circle (.075);
						\draw[very thin, dashed] (8,0) circle (.075);
						\fill (10,0) circle (.075);

						\draw (2,4) -- (4,0);
						\draw (2,4) -- (10,0);
						\draw (8,4) -- (10,0);
						\draw (8,4) -- (4,0);

						\draw plot[domain=-3:3] ({\x+5},{3.45+(\x)^2/(4^2)});
						\draw plot[domain=-3:3] ({\x+7},{.55-(\x)^2/(4^2)});

						\fill (3,2) node [left] {$\delta$};
						\fill (9,2) node [right] {$\beta$};
						\fill (5,1) node [left] {$\tau$};
						\fill (5,2.4) node [left] {$\sigma$};
						\fill (6,3.2) node {$\alpha$};
						\fill (7,.8) node {$\gamma$};
					\end{tikzpicture}
				\end{center}
				\caption{Replacing a bridging arc by another}\label{fig:mutation2}
			\end{figure}

			Assume first that $w$ is a coloured $m$-walk on $T$ in which $\sigma$ arises with a $+$ decoration. Replacing $\sigma^+$ by $\alpha^+ \tau^- \gamma^+$ and $\delta^+ \tau^- \beta^+$ provides two coloured $m$-walks on $T'$. We claim that every coloured $m$-walk on $T'$ including $\tau$ with a $-$ sign is obtained in this way. Fix thus a coloured $m$-walk on $T'$ with $\tau$ appearing with a $-$ sign. If $\tau$ is crossing from $o$ to $\iota$, then it must be of the form $\alpha^+ \tau^- \gamma^+$. If $\tau$ is crossing from $\iota$ to $o$ then consider the coloured $m$-walk on $T$ obtained by replacing $\tau^-$ by $\delta^- \sigma^+ \beta^-$. Then the coloured $m$-walk on $T'$ is obtained by replacing $\sigma^+$ by $\delta^+ \tau^- \beta^+$, which proves the claim. The same argument works for walks with $\tau$ having a $+$ decoration (and with $\sigma$ decorated with a $-$). 
		\end{proof}

		Thus, for any $\gamma \in \hat \A(C_{p,q})$ we simply write $x_\gamma$ for the element in $x^T_\gamma$ and for any collection $\Gamma$ of elements in $\hat \A(C_{p,q})$, we set $x_\Gamma = \prod_{\gamma \in \Gamma} x_\gamma$.

	\subsection{From loops to elements in $\mathcal B_Q$}
		\begin{lem}\label{lem:FmXdelta}
			For any $m \geq 1$, we have $x_{z_m} = F_m(X_\delta)$.
		\end{lem}
		\begin{proof}
			We first assume that $p+q \geq 3$ so that $Q$ is not the Kronecker quiver and by symmetry, we can assume that $p>1$.

			It follows from Lemma \ref{lem:mutXtz} that it is enough to prove the lemma for a particular choice of triangulation $T$. We fix an integer $m \geq 1$.

			We consider the following triangulation $T$ of $C_{p,q}$ (viewed in the universal cover):
			
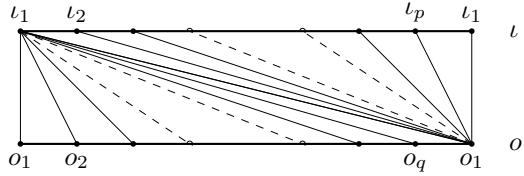
\begin{figure}[H]
				\begin{center}
					\begin{tikzpicture}[scale = .75]
						\draw[thick] (0,0) -- (8,0);
						\fill (8.5,2) node [right] {$\iota$};
						\draw[thick] (0,2) -- (8,2);
						\fill (8.5,0) node [right] {$o$};
							
						\fill (0,2) node [above] {$\iota_1$};
						\fill (1,2) node [above] {$\iota_2$};
						\fill (7,2) node [above] {$\iota_p$};
						\fill (8,2) node [above] {$\iota_1$};

						\fill (0,0) node [below] {$o_1$};
						\fill (1,0) node [below] {$o_2$};
						\fill (7,0) node [below] {$o_q$};
						\fill (8,0) node [below] {$o_1$};

						\foreach \x in {0,1,2,6,7,8}
						{
							\fill (\x,0) circle (.05);
							\fill (\x,2) circle (.05);
						}
						\foreach \x in {3,5}
						{
							\draw[very thin,dashed] (\x,0) circle (.05);
							\draw[very thin,dashed] (\x,2) circle (.05);
						}

						\foreach \x in {0,1,2,6,7,8}
						{
							\draw[very thin] (0,2) -- (\x,0);
							\draw[very thin] (8,0) -- (\x,2);
						}
						\foreach \x in {3,5}
						{
							\draw[very thin,dashed] (0,2) -- (\x,0);
							\draw[very thin,dashed] (8,0) -- (\x,2);
						}
					\end{tikzpicture}
				\end{center}
			\end{figure}

			We denote by $M_m$ the object in $\CC_Q$ corresponding to $\gamma_{\iota_1}^{(mp)}$, which, with the previous notations amounts to saying that $M_m = R_1^{(mp)}$ in the tube $\mathcal T_p$. We set $N_m = R_2^{(mp-1)}$. It follows from the so-called \emph{higher difference properties} proved in \cite[Proposition 3.3]{Dupont:transverse} that 
			$$F_m(X_{\delta}) = X^T_{M_m} - X^T_{N_m}.$$

			We now prove the analogous identity for the formula $x^T$. Identifying objects in $\mathcal T_p$ with arcs in $\C^\iota(C_{p,q})$, we are in the situation depicted in Figure \ref{fig:MandNasarcs}.
			\begin{figure}[H]
				\begin{center}
					\begin{tikzpicture}[scale = .7]
					
						\draw (0,0) -- (5,0);
						\draw (0,3) -- (5,3);

						\draw[very thin, dashed] (5,0) -- (10,0);
						\draw[very thin, dashed] (5,3) -- (10,3);

						\draw (10,0) -- (16,0);
						\draw (10,3) -- (16,3);

						\foreach \x in {0,5,...,15}
						{
							\fill (\x,0) circle (.075);
							\fill (\x,3) circle (.075);
						}
						\foreach \x in {1,11,16}
						{
							\fill (\x,0) circle (.075);
							\fill (\x,3) circle (.075);
						}
						\foreach \x in {4,14}
						{
							\fill (\x,0) circle (.075);
							\fill (\x,3) circle (.075);
						}

						\foreach \x in {2,3,6,7,8,9,12,13}
						{
							\draw[very thin, dashed] (\x,0) circle (.075);
							\draw[very thin, dashed] (\x,3) circle (.075);
						}

						\foreach \x in {0,5,10,15}
						{
							\fill (\x,0) node [below] {$o_1$};
							\fill (\x,3) node [above] {$\iota_1$};
						}

						\foreach \x in {0,10,15}
						{
							\fill (\x+1,0) node [below] {$o_2$};
							\fill (\x+1,3) node [above] {$\iota_2$};
						}

						\foreach \x in {5,15}
						{
							\fill (\x-1,0) node [below] {$o_{q}$};
							\fill (\x-1,3) node [above] {$\iota_{p}$};
						}

						\foreach \x in {0,15}
						{
							\draw (\x,0) -- (\x,3);
							\draw (\x,3) -- (\x+1,0);

							\fill (\x+.1,1.5) node [left] {$\alpha$};
							\fill (\x+.45,1.5) node [right] {$\beta$};

						}
						\foreach \x in {5,10}
						{
							\draw[very thin, dashed] (\x,0) -- (\x,3);
						}

						\foreach \x in {0,10}
						{
							\draw (\x,3) -- (\x+5,0);
							\draw (\x+1,3) -- (\x+5,0);
							\draw[very thin, dashed] (\x+2,3) -- (\x+5,0);
							\draw[very thin, dashed] (\x+3,3) -- (\x+5,0);
							\draw (\x+4,3) -- (\x+5,0);

							\draw (\x,3) -- (\x+1,0);
							\draw[very thin, dashed] (\x,3) -- (\x+2,0);
							\draw[very thin, dashed] (\x,3) -- (\x+3,0);
							\draw (\x,3) -- (\x+4,0);
						}
						\foreach \x in {5}
						{
							\draw[very thin,dashed] (\x,3) -- (\x+5,0);
							\draw[very thin,dashed] (\x+1,3) -- (\x+5,0);
							\draw[very thin,dashed] (\x+2,3) -- (\x+5,0);
							\draw[very thin,dashed] (\x+3,3) -- (\x+5,0);
							\draw[very thin,dashed] (\x+4,3) -- (\x+5,0);

							\draw[very thin,dashed] (\x,3) -- (\x+1,0);
							\draw[very thin,dashed] (\x,3) -- (\x+2,0);
							\draw[very thin,dashed] (\x,3) -- (\x+3,0);
							\draw[very thin,dashed] (\x,3) -- (\x+4,0);
						}

						\draw[thick] plot[domain=-8:8] ({\x+8},{2.25*(1-(\x)^2/(8^2))});
						\draw[thick] plot[domain=-7:7] ({\x+8},{1-(\x)^2/(7^2)});	

						\draw (8,1.25) node[above,fill=white] {$M_m$}; 
						\draw (8,0) node[above,fill=white] {$N_m$}; 
					\end{tikzpicture}
				\end{center}
				\caption{Arc representations of $M_m$ and $N_m$}\label{fig:MandNasarcs}
			\end{figure}
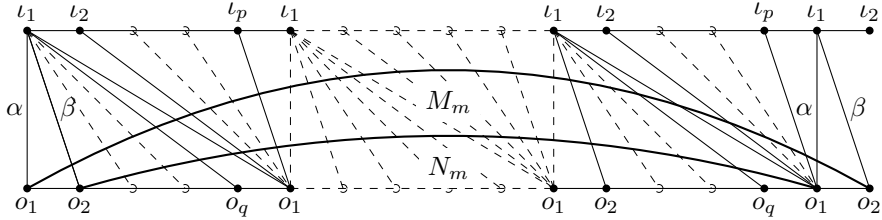

			As before, we denote by $\mathcal W^T_{N_m}$ the set of coloured walks on $T$ which are considered in the formula $x^T_{N_m}$ and by $\mathcal W^T_{M_m}$ the set of coloured walks on $T$ which are considered in the formula $x^T_{M_1}$. As in Figure \ref{fig:MandNasarcs}, we denote by $\alpha$ the edge of $T$ joining $o_1$ to $\iota_1$ and by $\beta$ the edge of $T$ joining $\iota_1$ to $o_2$. Since in the formula for $x^T$ boundary components contribute as 1, in order to simplify notations we will always denote by $o$ a boundary segment on $o$ and by $\iota$ a boundary segment on $\iota$.

			For any coloured walk $w$ in $\mathcal W^T_{N_m}$, the coloured walk $\alpha^+\beta^- w \alpha^- \beta^+$ is in $\mathcal W^T_{M_m}$ and the respective contributions in $x^T_{N_m}$ and $x^T_{M_m}$ are the same. 

			We denote by $\mathcal W^T_{z_m}$ the set of $m$-coloured walks on $T$ (i.e., those which are considered in the formula $x^T_{z_m}$). We denote by $\mathcal W^{T;\iota}_{z_m}$ the set of coloured walks in $\mathcal W^T_{z_m}$ going through the first lift of $\iota_1$ in the $m$-fold cover $C_{mp,mq}$ and by $\mathcal W^{T;o}_{z_m}$ its complement in $\mathcal W^T_{z_m}$ (which consists of walks passing through the first lift of $o_1$ but not all such walks). 

			If $w \in \mathcal W^{T;o}_{z_m}$ then $wo^+ \in \mathcal W^T_{M_m}$ and the respective contributions in $x^T_{z_m}$ and $x^T_{M_m}$ are the same.
			If $w \in \mathcal W^{T;\iota}_{z_m}$, then $o^+\beta^-w\beta^+ \in \mathcal W^T_{M_m}$ and the respective contributions in $x^T_{z_m}$ and $x^T_{M_m}$ are the same.
			Moreover, we have 
			$$\mathcal W^T_{M_m} = o^+\beta^-(\mathcal W^{T;\iota}_{z_m}) \beta^+ \sqcup (\mathcal W^{T;o}_{z_m}) o^+ \sqcup \alpha^+\beta^-(\mathcal W^T_{N_m}) \alpha^-\beta^+.$$

			Thus, 
			\begin{align*}
				x^T_{M_m} 
					& = \sum_{w \in \mathcal W^T_{M_m}} x(w) \\
					& = \sum_{w \in \mathcal W^{T;\iota}_{z_m}} x(o^+\beta^-w\beta^+) + \sum_{w \in \mathcal W^{T;o}_{z_m}} x(wo^+) + \sum_{w \in \mathcal W^T_{N_m}} x(\alpha^+\beta^-w\alpha^-\beta^+)\\
					& = \sum_{w \in \mathcal W^{T;\iota}_{z_m}} x(w) + \sum_{w \in \mathcal W^{T;o}_{z_m}} x(w) + \sum_{w \in \mathcal W^T_{N_m}} x(w)\\
					& = \left( \sum_{w \in \mathcal W^{T;\iota}_{z_m}} x(w) + \sum_{w \in \mathcal W^{T;o}_{z_m}} x(w) \right) + \sum_{w \in \mathcal W^T_{N_m}} x(w)\\
					& = x^T_{z_m} + x^T_{N_m}.
			\end{align*}

			Now, since $M_m$ and $N_m$ are curves in $\C^\iota(C_{p,q})$, we can apply Lemma \ref{lem:xTXTCp} and we get
			$$x^T_{z_m} = x^T_{M_m} - x^T_{N_m} = X^T_{M_m} - X^T_{N_m} = F_m(X_\delta)$$
			and the lemma is proved for $p+q \geq 3$.

			Assume now that $p=q=1$ and consider the same triangulation $T$ as in Figure \ref{fig:z1}. Consider the annulus $C_{2,1}$ equipped with the triangulation $T'$ containing two bridging edges $\alpha$ and $\beta$ and one peripheral edge $\tau$, as in Figure \ref{fig:C21} below. 
			\begin{figure}[H]
				\begin{center}
					\begin{tikzpicture}[scale = .75]
						\draw (0,2) -- (4,2);
						\draw (0,0) -- (4,0);

						\foreach \x in {0,2,4}
						{
							\fill (\x,0) circle (.05);
						}
						\foreach \x in {0,4}
						{
							\fill (\x,2) circle (.05);
						}

						\draw (0,2) -- (0,0);
						\draw (0,2) -- (4,0);
						\draw plot[domain=0:4] ({\x},{.25-(\x-2)^2/16});
						\draw (4,2) -- (4,0);

						\fill (0,1) node [left] {$\alpha$};
						\fill (4,1) node [right] {$\alpha$};
						\fill (2,1) node [above] {$\beta$};
						\fill (2,.25) node [above] {$\tau$};
					\end{tikzpicture}
				\end{center}
				\caption{The triangulation $T'$ of $C_{2,1}$.}\label{fig:C21}
			\end{figure}
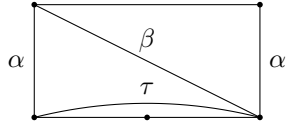
			Consider the ring homomorphism $\pi$ from the ring of Laurent polynomials in the cluster $T'$ of $\mathcal A_{C_{2,1}}$ to the ring of Laurent polynomials in the cluster $T$ of $\mathcal A_{C_{1,1}}$ sending $x_\alpha$ to $x_\alpha$, $x_\beta$ to $x_\beta$ and $x_\tau$ to 1. We denote by $\tilde z_m$, with $m \geq 1$ the loops in $C_{2,1}$. It follows that $\pi(x^{T'}_{\tilde z_m}) = x^T_{z_m}$ for any $m \geq 1$. In particular, it follows from the above discussion that 
			$$x^T_{z_m} = \pi(x^{T'}_{\tilde z_m}) = \pi(F_m(x^{T'}_{\tilde z})) = F_m(\pi(x^{T'}_{\tilde z})) = F_m(x^T_{z})$$
			which finishes the proof.
		\end{proof}

		\begin{exmp}
			We illustrate the combinatorial interpretation of these higher difference properties for $m=1$ in the following example of type $\widetilde A_{1,3}$. Consider the following triangulation $T$ of $C_{1,3}$:
			
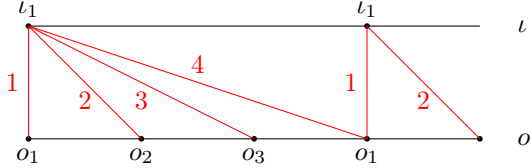
\begin{figure}[H]
				\begin{center}
					\begin{tikzpicture}[scale = .75]
						\draw (0,0) -- (8,0);
						\fill (8.5,2) node [right] {$\iota$};
						\draw (0,2) -- (8,2);
						\fill (8.5,0) node [right] {$o$};
							
						\foreach \x in {0,6}
						{
							\fill (\x,2) circle (.05);
						}
						\foreach \x in {0,2,...,8}
						{
							\fill (\x,0) circle (.05);
						}
						
						\fill (0,0) node[below] {$o_1$};
						\fill (2,0) node[below] {$o_2$};
						\fill (4,0) node[below] {$o_3$};
						\fill (0,2) node[above] {$\iota_1$};
						\fill (6,0) node[below] {$o_1$};
						\fill (6,2) node[above] {$\iota_1$};

						\foreach \x in {0,2,...,6}
						{
							\draw[red] (0,2) -- (\x,0);
						}
						\draw[red] (6,2) -- (6,0);
						\draw[red] (6,2) -- (8,0);

						\fill[red] (0,1) node[left] {1};
						\fill[red] (1,1) node[below] {2};
						\fill[red] (2,1) node[below] {3};
						\fill[red] (3,1) node[above] {4};
						\fill[red] (6,1) node[left] {1};
						\fill[red] (7,1) node[below] {2};
					\end{tikzpicture}
				\end{center}
			\end{figure}
			The quiver $Q$ of the triangulation $T$ is the following quiver of type $\widetilde A_{3,1}$:
			$$
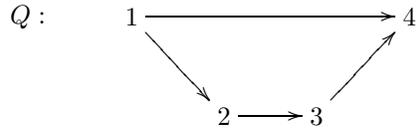
\xymatrix{
				Q : & 1\ar[rd] \ar[rrr] &&& 4 \\
					&& 2 \ar[r] & 3 \ar[ru]
			}$$

			Now, consider the following arcs $M_1$ and $N_1$:
			
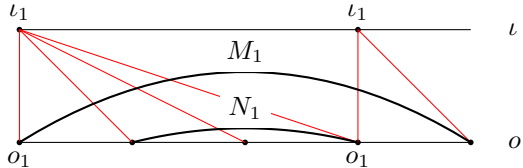
\begin{figure}[H]
				\begin{center}
					\begin{tikzpicture}[scale = .75]
						\draw (0,0) -- (8,0);
						\fill (8.5,0) node [right] {$o$};
						\draw (0,2) -- (8,2);
						\fill (8.5,2) node [right] {$\iota$};
							
						\foreach \x in {0,6}
						{
							\fill (\x,2) circle (.05);
						}
						\foreach \x in {0,2,...,8}
						{
							\fill (\x,0) circle (.05);
						}
						
						\fill (0,0) node[below] {$o_1$};
						\fill (0,2) node[above] {$\iota_1$};
						\fill (6,0) node[below] {$o_1$};
						\fill (6,2) node[above] {$\iota_1$};

						\foreach \x in {0,2,...,6}
						{
							\draw[red] (0,2) -- (\x,0);
						}
						\draw[red] (6,2) -- (6,0);
						\draw[red] (6,2) -- (8,0);

						\draw[thick] plot[domain=0:8] ({\x},{1.25*(1-(\x-4)^2/(4^2))});
						\draw[thick] plot[domain=2:6] ({\x},{0-(\x-4)^2/(16)+.25});

						\draw (4,1.25) node[above,fill=white] {$M_1$}; 
						\draw (4,.25) node[above,fill=white] {$N_1$}; 
					\end{tikzpicture}
				\end{center}
			\end{figure}
			The corresponding representations of $Q$ (viewed here as objects in the cluster category $\CC_Q$), are 
			$$\xymatrix{
				M_1 : & \k \ar[rd]^0 \ar[rrr]^{1_\k} &&& \k, \\
					&& \k \ar[r]^{1_\k} & \k \ar[ru]^{1_\k} 
			}
			\quad
			\xymatrix{
				N_1 : & 0\ar[rd] \ar[rrr] &&& 0. \\
					&& 0 \ar[r] & \k \ar[ru]
			}
			$$

			We now list the coloured walks corresponding to $M_1$, the monomials that they give rise to in the formula $x_{M_1}$ and the dimension vector $\mathbf e$ giving the same Laurent monomial in the formula $X_{M_1}$. We also make explicit the factorisations in terms of $\mathcal W^T_{N_1}$ and $\mathcal W^T_z$:
			$$\begin{array}{|c|c|c|}
				\hline
				w \in \mathcal W^T_{M_1} & x(w) & \mathbf e \\
				\hline 
				(1^+4^-) o^+ & x_1/x_4 & (0000) \\
				(1^+3^-o^+4^-\iota^+1^-) o^+ & 1/x_3x_4 & (0001) \\
				(1^+2^-o^+3^-\iota^+1^-) o^+ & 1/x_2x_3 & (0011) \\
				(o^+2^-\iota^+1^-) o^+ & 1/x_1x_2 & (0111) \\
				o^+2^-(4^+1^-)2^+ & x_4/x_1 & (1111) \\
				1^+2^-(2^+3^-o^+)1^-2^+ & x_2/x_3 & (1001) \\
				1^+2^-(o^+3^-4^+)1^-2^+ & x_4/x_3 & (1011) \\
				\hline
			\end{array}$$
			Similarly for $N_1$ and $z$, we have 
			$$\begin{array}{|c|c|c|}
				\hline
				w \in \mathcal W^T_{N_1} & x(w) & \mathbf e \\
				\hline 
				2^+3^-o^+ & x_2/x_3 & (0000) \\
				o^+3^-4^+ & x_4/x_3 & (0010) \\
				\hline
			\end{array}$$
			and for $z$, 
			$$\begin{array}{|c|c|c|}
				\hline
				w \in \mathcal W^T_z & x(w) & \mathbf e \\
				\hline 
				1^+4^- & x_1/x_4 & (0000) \\
				1^+3^-o^+4^-\iota^+1^- & 1/x_3x_4 & (0001) \\
				1^+2^-o^+3^-\iota^+1^- & 1/x_2x_3 & (0011) \\
				o^+2^-\iota^+1^- & 1/x_1x_2 & (0111) \\
				4^+1^- & x_4/x_1 & (1111) \\
				\hline
			\end{array}$$
			where the first four lines in the last table correspond to $\mathcal W^{T;o}_z$ and the last one to $\mathcal W^{T;\iota}_z$.
		\end{exmp}

	\subsection{Proof of Theorem \ref{theorem:BQgeometric}}
		We now finish the proof of Theorem \ref{theorem:BQgeometric}. Recall from the definition that
		$$\mathcal B_Q = \mathcal M_Q \sqcup \ens{X_R F_m(X_\delta) \ | \ m \geq 1, R \in \rr \kQ}.$$
		As we already mentioned, we know that 
		$$\mathcal M_Q = \ens{x_\Gamma \ | \ \Gamma \in \T(C_{p,q})}$$
		and that for any $M=\bigoplus_{i \in I} M_i \in \reg \kQ$, we have $X_M = x_\Gamma$ where $\Gamma$ is the family of arcs in $\C^\iota(C_{p,q}) \sqcup \C^o(C_{p,q})$ corresponding to $\ens{M_i}_{i \in I}$. Moreover, it is known that $M$ is rigid if and only if the corresponding curves in $\C_{p,q}$ do not intersect, see for instance \cite{BZ:clustercatsurfaces}. Therefore, to finish the proof of Theorem \ref{theorem:BQgeometric}, it is enough to observe that Lemma \ref{lem:FmXdelta} implies that $F_m(X_\delta) = x_{z_m}$ for any $m \geq 1$. \hfill \qed

\section{Proof of Theorem \ref{theorem:BQatomic}}\label{section:BQatomic}
	In this section $Q$ still denotes a quiver of type $\widetilde A_{p,q}$ with $p,q \geq 1$.

	Proving Theorem \ref{theorem:BQatomic} is equivalent to proving the following three points:
	\begin{enumerate}
		\item[\textbf{(B1)}] $\mathcal B_Q$ is a $\Z$-linear basis of $\mathcal A_Q$;
		\item[\textbf{(B2)}] $\mathcal B_Q$ is contained in the positive cone of $\mathcal A_Q$;
		\item[\textbf{(B3)}] Every element in the positive cone of $\mathcal A_Q$ can be written as a $\Z_{\geq 0}$-linear combination of elements of $\mathcal B_Q$.
 	\end{enumerate}

	\subsection{Proof of \textbf{(B1)}}
		It follows from \cite[Theorem 4.21]{Dupont:BaseAaffine} (see also \cite{GLS:generic}) that the set
		$$\mathcal G_Q = \mathcal M_Q \sqcup \ens{X_R X_\delta^m \ | \ m \geq 1, R \in \rr\kQ}$$
		is a $\Z$-linear basis of $\mathcal A_Q$. Since for any $m \geq 1$, $F_m$ is a monic polynomial of degree $m$, it easily follows that $\mathcal B_Q$ is a $\Z$-linear basis of $\mathcal A_Q$, see \cite[\S 6]{Dupont:genericvariables} for a precise description of the base change. This proves \textbf{(B1)}. \hfill \qed

	\subsection{Proof of \textbf{(B2)}}
		We need to prove that for any cluster $T$ in $\mathcal A_Q$, the Laurent expansion of any element of $\mathcal B_Q$ in the cluster $T$ has positive coefficients. An element $b \in \mathcal B_Q$ is a certain $x_\Gamma$ where $\Gamma$ is a collection of elements in $\hat \A(C_{p,q})$ and its $T$-expansion is given explicitly by the formula $x^T_\Gamma$. Now it follows from the definition of the map $x^T_?$ that $x^T_\gamma$ is a subtraction-free Laurent polynomial in the cluster $T$ for any $\gamma \in \hat \A(C_{p,q})$. Since subtraction-free Laurent polynomials in $T$ form a semiring, it follows that $x^T_\Gamma$ is also a subtraction-free Laurent polynomial in $T$. This proves \textbf{(B2)}. \hfill \qed

	\subsection{Proof of \textbf{(B3)}}
		This is the long part of the proof and it will be divided in several intermediate results. 

		\subsubsection{Beginning the proof}
			Let $y$ be a positive element in the cluster algebra $\mathcal A_Q$. According to \textbf{(B1)}, we can write 
			$$y = \sum_{\Gamma \in \hat \T(C_{p,q})} \lambda_{\Gamma}(y) x_\Gamma$$
			with $\lambda_{\Gamma}(y) \in \Z$. We need to prove that $\lambda_{\Gamma}(y) \geq 0$ for any $\Gamma \in \hat \T(C_{p,q})$. 

			When $\Gamma\in \T(C_{p,q})$ the situation is slightly easier, so we explain that situation first. We find some cluster $T$ which is compatible with $\Gamma$, and such that $x_\Gamma$ does not appear in the $T$-expansion of any other element $x_\Sigma$ of $\mathcal B_Q$. Then, $\lambda_\Gamma(y)$ coincides with the coefficient of $x_\Gamma$ in the $T$-expansion of $y$, which is non-negative by assumption. One point to bear in mind about this strategy is that the same $T$ must work for all choices of $\Sigma$. 

			When $\Gamma\in \hat \T(C_{p,q})\setminus \T(C_{p,q})$, the situation is more complicated for two reasons. First of all, there are infinitely many clusters compatible with $\Gamma$, and it is not enough to pick one. Rather, we define an infinite family of clusters $T_r$ for $r \in \mathbb Z$, any of which is compatible with $\Gamma_r$.

			Secondly, the expansion of $x_\Gamma$ in any cluster will include multiple terms. We will therefore pick one, and call it $t_{\Gamma,r}$. For any $r$, $t_{\Gamma,r}$ will appear in $T_r$-expansion of $x_\Gamma$ with coefficient one.

			We will then show that, for any $\Sigma\in \hat T(C_{p,q})$, if $r$ is sufficiently large, then $t_{\Gamma,r}$ does not appear in the $T_r$ expansion of $x_\Sigma$. Since, in our sum for $y$, only finitely many terms appear, it follows that we can choose $r$ large enough for all possible $\Sigma$ 
			simultaneously. Then the coefficient of $x_\Gamma$ in the cluster expansion of $y$ agrees with the coefficient of $t_{\Gamma,r}$ in the $T_r$-expansion of $y$, which is positive, so we are done. 

			The remainder of the proof fills in the details of the discussion above. First we treat the case that $\Gamma\in \T(C_{p,q})$, and then we treat the case $\Gamma \in \hat \T(C_{p,q})\setminus \T(C_{p,q})$. Within each of these cases, there are two subcases, depending on whether $\Sigma$ belongs to $\hat \T(C_{p,q}) \setminus \T(C_{p,q})$.

		\subsubsection{Technical lemmas}
			Before we can complete the proof, we need to collect some technical lemmas.

			\begin{lem}\label{lemma:techone}
				Fix $T$ a triangulation of $C_{p,q}$, and let $\gamma$ be a peripheral arc in $C_{p,q}$, with $\gamma \not\in T$. Then any term in the $T$-expansion of $x_\gamma$ has negative degree with respect to the cluster variables corresponding to arcs of $T$ which cross $\gamma$. 
			\end{lem}
			\begin{proof} 
				The proof is essentially the same as that of Lemma \ref{lemma:Atechone}: for any coloured $\gamma$-walk $w$ on $T$, the first and last segments of $w$ run along arcs which do not cross $\gamma$, and therefore do not contribute to the degree of the corresponding term $p(w)$.  
			\end{proof}

			We now prove an analogue of Lemma \ref{lemma:Atech}.
			\begin{lem} \label{lemma:periph}
				Fix a triangulation $T$ of $C_{p,q}$. Let $\gamma$ be a peripheral arc in $C_{p,q}$. Suppose that $\beta$ is an arc compatible with $\gamma$. Then each term in the $T$-expansion of $\beta$ has non-positive degree with respect to the set of edges which cross $\gamma$. 
			\end{lem}
			\begin{proof}
				It is possible that the beginning and ending segments of $\gamma$ lie in the same triangle of $T$: this only happens if $\gamma$ is homotopic to the entire inner boundary or the entire outer boundary. In this case, we call $\gamma$ a \emph{near-loop}. 

				Suppose first that $\gamma$ is not a near-loop. For each lifting of $\gamma$ to an arc $\widetilde \gamma$ of the universal cover of $C_{p,q}$, define the polygon $P_{\widetilde \gamma}$ as in the proof of Lemma \ref{lemma:Atech}. These polygons do not overlap (except possibly at a vertex). 

				Let $w$ be a coloured $\beta$-walk on $T$. The argument from Lemma \ref{lemma:Atech} goes through --- in order for the degree of the term corresponding to $w$ to be positive with respect to the edges crossed by $\gamma$, there would have to be some $\widetilde \gamma$ such that $w$ crosses $\widetilde \gamma$ an odd number of times, which is impossible. 

				Suppose next that $\gamma$ is a near-loop. For convenience, fix that $\gamma$ is attached to the inside boundary component. There are three possibilities
				for $\beta$: it might be peripheral on the inside, peripheral on the outside, or bridging. 

				Suppose first that $\beta$ is peripheral on the inside. The proof of Lemma \ref{lemma:Atech} goes through without any changes. 

				Suppose next that $\beta$ is peripheral on the outside. Let $w$ be a $T$-walk for $\beta$. Define $P$ to consist of the union of those triangles
				through which $\gamma$ passes. Note that this includes all triangles of $T$ which have vertices lying on both boundary components. As in Lemma \ref{lemma:Atech}, the even-numbered edges of $w$ which cross $\beta$ form a consecutive string, say $w_{2i},w_{2i+2},\dots, w_{2j}$. In order for the degree of the corresponding term $p(w)$ to be positive, it must be the case that all the odd-numbered edges from $w_{2i-1}$ and $w_{2j+1}$ also cross $\gamma$, resulting in a subsequence $\overline w=w_{2i-1},\dots,w_{2j+1} $ of $w$ where all the steps of $\overline w$ cross $\gamma$, while the step before and the one after do not. Note also that $\overline w$ begins on the outside component (since $w_{2i-2}$ crosses $\beta$ but not $\gamma$). The one extra observation to make in this case is that there are edges of $T$ which run from the inside component to the inside component, and cross $\gamma$ twice. (This type of phenomenon does not arise in a disc.) Such an edge would foil the argument of Lemma \ref{lemma:Atech}, which depends on the parity of the total number of crossings. However, such an edge cannot appear in even position in $\overline w$, because it does not cross $\beta$. We also claim that it cannot appear in odd position in $\overline w$. Suppose it appears in position $2k+1$. Up until that point, we assume that each edge $w_{2i-1},\dots,w_{2k-1}$ crosses $\gamma$ once --- but this means that $w_{2k+1}$ begins on the outside edge, so it cannot be one of these pathological edges. 

				Suppose next that $\beta$ is bridging. Since $\gamma$ is a near-loop, and $\beta$ is compatible with $\gamma$, $\beta$ shares an endpoint with $\gamma$. The argument from the proof of Lemma \ref{lemma:Atech} for the case that $\beta$ and $\gamma$ share an endpoint goes through without changes. 
			\end{proof}

			We prove a similar lemma, involving a peripheral arc $\gamma$ and a loop. 
			\begin{lem} \label{lemma:periph-loop}
				Fix a triangulation $T$ of $C_{p,q}$. Let $\gamma$ be an arc in $\A(C_{p,q})$ which is peripheral. For any positive integer $m$, each term in the $T$-expansion of $x_{z_m}$ has non-positive degree with respect to the set of edges which cross $\gamma$. 
			\end{lem}
			\begin{proof}
				Consider a term in the $T$-expansion of $x_{z_m}$, and let $w$ be the corresponding walk. If all the even-numbered edges of $w$ cross $\gamma$, 
				then the result is true, so suppose otherwise. 

				Consider a maximal-length subsequences of even-numbered edges of $w$ which cross $\gamma$, say $w_{2i},\dots,w_{2j}$. It will suffice to show that it is impossible to have both $w_{2i-1}$ and $w_{2j+1}$ crossing $\gamma$. Suppose both of them do cross $\gamma$. It would follow the endpoints of the walk $w_{2i-1}\dots w_{2j+1}$ are on opposite sides of $\gamma$. Therefore, one of the endpoints is in the region cut out by $\gamma$ and the boundary component to which $\gamma$ is connected. But then the next even-numbered edge, which crosses from one boundary component to the other, necessarily also crosses $\gamma$, contrary to our assumption. 
			\end{proof}

			Before we go on, we also recall a representation-theoretic lemma from \cite{Cerulli:finitetype}, and for this, we need to recall some additional notions. A fuller account of them can be found in \cite{Cerulli:finitetype} and in \cite{DWZ:potentials, DWZ:potentials2} from which it draws. 

			Let $T$ be a triangulation of $\C_{p,q}$, whose arcs are numbered 1 to $n$. Associated to it is a quiver $Q$, whose vertices are associated to the arcs of $T$, 
			a potential $S$, and a collection of (complex) {\it decorated quiver representations} $M_\alpha$, where $\alpha$ ranges over the isotopy classes of curves in $\C_{p,q}$. For us, the decorations will not play a role, so we will not distinguish the decorated quiver representation from the representation itself. 

			The cluster variable associated to $\alpha$ can be recovered in the following way. First, for any $\e$ in $\mathbb Z^n$, thought of as a dimension vector, we can consider the $\Gr_{\e}(M)$, the Grassmannian of subrepresentations of $M$ whose dimension at vertex $i$ is $\e_i$. Then the $T$-expansion of $x_\alpha$
			can be expressed as follows:
			$$x_\alpha= \sum_{\e} \chi(\Gr_{\e}(M_\alpha)) \x^{\g_\alpha+B\e}$$ 

			Here $\chi$ denotes the Euler-Poincar\'e characteristic, $B$ is the $B$-matrix corresponding to $T$, and $\g_\alpha$ is the \emph{$\g$-vector} associated to $\alpha$. For $\v\in \mathbb Z^n$, we write $\x^\v$ for $x_1^{\v_1}\dots x_n^{\v_n}$, where $x_1,\dots,x_n$ are the cluster variables associated to the arcs of $T$. Inversely, given a monomial $\x^\v$, we let $\exp(\x^\v)=\v$. 

			A formula of a similar form can be given for any cluster monomial. Let $\Sigma\in \T(C_{p,q})$. Define $M_\Sigma$ to be the sum of the corresponding $M_\alpha$'s, with multiplicity, and define $\g_\Sigma$ to be the sum of the corresponding $\g_\alpha$'s. Then
			$$x_\Sigma= \sum_{\e} \chi(\Gr_{\e}(M_\Sigma)) \prod_{i=1}^n \x^{\g_\Sigma+B\e}$$
			We write $x_{\Sigma}(\e)$ for the term in the above sum corresponding to $\e$.

			We endow $\Z^n$ with the standard inner product given by 
			$$\e.\f = \sum_{i=1}^n e_if_i$$
			for any $\e = (e_i)_{1 \leq i \leq n}$ and $\f = (f_i)_{1 \leq i \leq n}$ in $\Z^n$.

			The following lemma was proved in \cite[Lemma 5.1]{Cerulli:finitetype} for finite type cluster algebras. Its proof remains valid in the context of cluster algebras of type $\widetilde A$: condition (12) of \cite{Cerulli:finitetype} is satisfied by \cite[Theorem 36]{Labardini:potentialssurfaces} or (in a more special case sufficient for our purposes) \cite[Lemma 2.4]{ABCP}.

			\begin{lem}[\cite{Cerulli:finitetype}]\label{lem:cerulli}
				Let $T$ be a triangulation of $C_{p,q}$, and let $\Sigma \in \T(C_{p,q})$, such that none of the arcs in $\Sigma$ is in $T$. Then for any $\mathbf e \neq 0$ such that $\Gr_{\e}(M_\Sigma) \neq \emptyset$, we have:
				$$\e.\exp(x_{\Sigma}(\e)) < 0.$$ \hfill \qed
			\end{lem}

			We also need the following proposition from \cite{DWZ:potentials2}. Let $\Sigma\in \T(C_{p,q})$ such that it shares no arcs with $T$. Then define $E(M_\Sigma)=\dim(M_\Sigma).\g_\Sigma + \dim \Hom(M_\Sigma,M_\Sigma)$. 

			\begin{prop}[{\cite[Corollary 7.2]{DWZ:potentials2}}]\label{Ezero} 
				For any $\Sigma\in \T(C_{p,q})$ containing no arcs from $T$, we have that $E(M_\Sigma)=0$.
			\end{prop}

			We can now complete our proof.

		\subsubsection{The case when $\Gamma\in \T(C_{p,q})$}
			In this case, $x_\Gamma$ is a cluster monomial. Depending on $\Gamma$, there may or may not be any choice for a triangulation $T$ compatible with $\Gamma$, but in any case, we choose it arbitrarily. 

			If $\Gamma$ and $\Sigma$ have arcs in common, we can remove corresponding arcs from each, and $x_\Gamma$ will appear in the $T$-expansion of $x_\Sigma$ iff this is true after the cancellation. So we may assume that $\Gamma$ and $\Sigma$ have no arcs in common.

			\subsubsubsection{The case where $\Gamma \in \T(C_{p,q})$ and $\Sigma \in \T(C_{p,q})$}
				In this case, both $\Gamma$ and $\Sigma$ are cluster monomials and the argument is the same as in \cite{Cerulli:finitetype}, but we give it for completeness. 

				Suppose $\Sigma$ contains an arc $\sigma$ of $T$. By assumption, this arc does not appear in $\Gamma$. Snip the annulus open along $\sigma$. This results in either a disc or a disc together with an annulus with fewer marked points than before. Specializing the cluster variable $x_\sigma$ 
				to one and applying induction to the annulus if necessary, we deduce that $\Sigma$ and $\Gamma$ coincide in the interior of the new surface(s). This implies that $\Sigma$ and $\Gamma$ coincide except that $\Sigma$ contains an extra arc not appearing in $\Gamma$, which is obviously impossible. 

				We may therefore assume that $\Sigma$ contains no arcs of $T$. Thus, it follows from Lemma \ref{lem:cerulli} that $X^T_M$ is the sum of $\x^T_M(0)$ and of proper Laurent monomials in the cluster $T$, since $\e.\exp(x_\Sigma(\e))<0$ implies that some term of 
				$\exp(x_\Sigma(\e))<0$. 

				It thus remains to prove that $\x_\Sigma(0)$ is a proper Laurent monomial. As in \cite{Cerulli:finitetype}, this is equivalent to showing that $\g_\Sigma$ has at least one negative entry. By Proposition \ref{Ezero}, we know $0=E(M_\Sigma)=\dim M.\g_\Sigma + \dim\Hom(M_\Sigma,M_\Sigma)$. Since the second term is strictly positive (except in the degenerate case that $M_\Sigma=0$), the first is negative, implying that $\g_\Sigma$ has at least one negative term, as desired. 

				Therefore, $x_\Sigma$ is a sum of proper Laurent monomials in the cluster $T$ and $x_\Gamma$ is not a summand of the $T$-expansion of $x_\Sigma$.

			\subsubsubsection{The case where $\Gamma \in \T(C_{p,q})$ and $\Sigma \in \hat \T(C_{p,q}) \setminus \T(C_{p,q})$}
				Since $\Sigma \in \hat \T(C_{p,q}) \setminus \T(C_{p,q})$, it cannot contain any bridging edges. 

				If all the arcs of $\Sigma$ were in $T$, then, since all the terms in the expansion of the loop have bridging edges of $T$ in the denominator, the same will be true for all the terms of $x_\Sigma$, since all the arcs of $\Sigma$ are necessarily peripheral. 

				Suppose now that $\Sigma$ has a peripheral edge $\sigma$ which is not in $T$. By Lemmas \ref{lemma:techone}, \ref{lemma:periph}, and \ref{lemma:periph-loop}, the total degree of any term in the expansion of $\Sigma$ with respect to the edges of $T$ which cross $\sigma$, is negative. It follows that $x_\Gamma$ does not appear in the $T$-expansion of $x_\Sigma$. 

				This completes the proof that $\lambda_\Gamma(y)$, the coefficient of $x_\Gamma$ in the expansion of $y$, is non-negative, in the case where $x_\Gamma$ is a cluster monomial. 

		\subsubsection{The case where $\Gamma \in \hat \T(C_{p,q}) \setminus \T(C_{p,q})$}
			We write $\Gamma = \ens{z_m} \sqcup \overline \Gamma$ for some $m \geq 1$ and some $\overline \Gamma \in \T(C_{p,q})$. 

			Since $\overline\Gamma$ contains no bridging arcs, there is some freedom to choose a triangulation of $C_{p,q}$. First of all, add enough peripheral edges to the edges of $C_{p,q}$ so that the region which is not triangulated is a subannulus which has one vertex on each boundary component. Denote these vertices by $O$ and $I$. 

			Now, we are going to define a $\mathbb Z$-indexed family of triangulations. For $r \in \mathbb Z$, define $T^{(r)}$ to be the triangulation obtained by adding an edge $\alpha$ which starts at $I$, wraps $r$ times around the annulus, then goes to $O$, and let $\beta$ be the edge which starts at $O$ and wraps $-r+1$ times around the annulus before returning to $I$. These two edges are compatible and define a triangulation. 

			The $T^{(r)}$-expansion of $x_\Gamma$ is $x_{\overline\Gamma}$ times the $T^{(r)}$-expansion of $x_{z_m}$. We note that the $T^{(r)}$ expansion of 
			$x_{z_m}$ contains a term $x_\beta^m/x_\alpha^m$. We define $t_{\Gamma,r}=x_{\overline\Gamma}x_\beta^m/x_\alpha^m$.

			\subsubsubsection{The case where $\Gamma \in \hat \T(C_{p,q}) \setminus \T(C_{p,q})$ and $\Sigma \in \T(C_{p,q})$}
				If $\Sigma$ has any arcs in common with $T$, then we can proceed as in the case that $\Gamma$ is a cluster monomial, so we may assume that $\Sigma$ has no such arcs. Therefore, for any $\e \neq 0$ in $\Z^n$ such that $\Gr_{\e}(M_\Sigma) \neq \emptyset$, we know from Lemma \ref{lem:cerulli} that $\e.\exp(\x_\Sigma(\e)) <0$. 

				We claim that for any such $\e$, we have $\e.\exp(t_{\Gamma,r}) \geq 0$. Clearly $\e.\exp(x_{\overline\Gamma}) \geq 0$, since all the entries in both vectors are non-negative. It is therefore enough to show that $\e.\exp(x_\beta^m/x_\alpha^m) \geq 0$. Write $Q_{T^{(r)}}$ for the quiver associated to $T^{(r)}$, and write $\widetilde Q_{T^{(r)}}$ for its full subquiver on the vertices $v_\alpha, v_\beta$ corresponding to $\alpha$ and $\beta$. This is a Kronecker quiver with two arrows from $v_\alpha$ to $v_\beta$. For $M$ a representation of $Q_{T^{(r)}}$, write $\widetilde M$ for the representation restricted to the subquiver. 

For $\sigma \in \Sigma$, consider $\widetilde M_\sigma$.  By \cite{BZ:clustercatsurfaces}, it is indecomposable, and its dimensions at $v_\alpha,v_\beta$ count
the number of intersections of $\sigma$ with $\alpha,\beta$.  By choosing
$r$ large enough, we may guarantee that $\sigma$ crosses $\beta$ at least
as many times as $\alpha$.  This implies that $\widetilde M_\sigma$ is 
preprojective or regular, which implies that the same is true of any
of its indecomposable subobjects.  

				Consider the Grothendieck group for representations of $\widetilde Q$, which we think of as $\mathbb Z^2$, by fixing the basis $[S_\alpha],[S_\beta]$.  We also think of this as the multiplicative group of monomials in $x_\alpha,x_\beta$. Then $\exp(x_\beta^m/x_\alpha^m)=(-m,m)$, which is $-m$ times the class of the null root for the Kronecker quiver. 
Therefore, if $\dd$ is a dimension vector of an indecomposable preprojective $\widetilde Q_{T^{(r)}}$-representation, $\dd.(-m,m)=m$, while if $\dd$ is the dimension vector of a regular indecomposable $\widetilde Q_{T^{(r)}}$-representation, $\dd.\exp(x^{m}_\alpha/x^m_\beta)=0$. It follows that, for any $\e$ such that 
$\Gr_\e(M_\Sigma) \ne \emptyset$, we have $\e.\exp(t_{\Gamma,r}) \geq 0$ for 
$r$ sufficiently large.  

				This shows that the only term in the $T$-expansion of $x_\Sigma$ which could coincide with $t_{\Gamma,r}$ is $x_\Sigma(0)$. To treat the case $\e=0$, we also follow \cite{Cerulli:finitetype}. Specifically, we show that $\dim M_\Sigma.\exp(t_{\Gamma,r})\geq 0$, while $\dim M_\Sigma.\exp(\x^{T^{(r)}}_M(0))<0$.

				As in the $\e \ne 0$ case, the first of these statements reduces to showing that $\dim M.\exp(x_\beta^m/x_\alpha^m)\geq 0$, and this will hold for $r$ sufficiently large, because we can arrange $\widetilde M$ to be preinjective or regular. For the second statement, we see that $\dim M_\Sigma. \exp(\x_\Sigma(0))=\dim M_\Sigma.\g_\Sigma$. Proposition \ref{Ezero} tells us that $0 = E(M) = \dim M_\Sigma.g_\Sigma + \dim \Hom(M,M).$ 
				Thus $\dim M.\exp(\x_\Sigma(0)) < 0$ as desired.

			\subsubsubsection{The case where $\Gamma \in \hat \T(C_{p,q}) \setminus \T(C_{p,q})$ and $\Sigma\in \hat \T(C_{p,q}) \setminus \T(C_{p,q})$}
				Note that in this case $\Sigma$ does not contain any bridging arc. As before, we may assume that $\Gamma$ and $\Sigma$ do not contain any arcs in common. 

				Suppose that $\Sigma$ contains some arc which does not appear in $T$. We can therefore apply Lemma \ref{lemma:periph-loop} to conclude that any term in the expansion of $x_\Sigma$ has negative degree with respect to arcs that cross $\gamma$.  Therefore $t_{\Gamma,r}$ cannot appear as such a term, since $x_\beta^m/x_\alpha^m$ has zero degree with respect to edges of $T$ which cross $\gamma$ (either both or neither of $\alpha$ and $\beta$ cross $\gamma$), and $x_{\overline\Gamma}$ has non-negative degree with respect to arcs that cross $\gamma$. 

				On the other hand, if $\Sigma$ contains only a loop and edges from $T^{(r)}$ (disjoint from those of $\overline\Gamma$), the claim is clear. Thus this case is established. 

				This completes the proof that $\lambda_\Gamma(y)$, the coefficient of $x_\Gamma$ in the expansion of $y$, is non-negative, in the case that $\Gamma$ includes a loop. We have therefore completed the proof of \textbf{(B3)}, and thus the proof of Theorem \ref{theorem:BQatomic}. \hfill \qed

\section{An example in type $\widetilde A_{2,2}$}
	In this short section, we give an explicit description of the atomic basis in a cluster algebra $\mathcal A$ of type $\widetilde A_{2,2}$. Such a cluster algebra is associated to the annulus $C_{2,2}$ with two marked points on each boundary component. We denote by $\mathcal M$ the set of cluster monomials in $\mathcal A$. Moreover, we distinguish four particular elements in the cluster algebra $\mathcal A$ which correspond to the following curves:
	
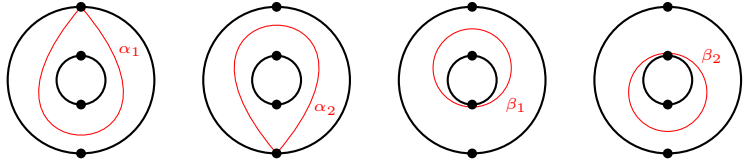
\begin{center}
		\begin{tikzpicture}[scale = .65]
		\tikzstyle{every node}=[font=\tiny]

		\draw[red] (0,1.5) .. controls (-3,-2) and (3,-2) .. (0,1.5);
		\fill[red] (1,.6) node {$\alpha_1$};

		\draw[red] (4,-1.5) .. controls (1,2) and (7,2) .. (4,-1.5);
		\fill[red] (5,-.6) node {$\alpha_2$};

		\draw[red] (8,.25) circle (.8); 
 		\fill[red] (8.9,-.5) node {$\beta_1$};

		\draw[red] (12,-.25) circle (.8); 
 		\fill[red] (12.9,.5) node {$\beta_2$};

		\foreach \x in {0,4,...,12}
		{
			\draw[thick] (\x,0) circle (.5);
			\draw[thick] (\x,0) circle (1.5);
			\fill (\x,.5) circle (.1);
			\fill (\x,-.5) circle (.1);
			\fill (\x,1.5) circle (.1);
			\fill (\x,-1.5) circle (.1);
		}
		\end{tikzpicture}
	\end{center}
	
	In terms of representation theory, these four curves correspond to the four indecomposable rigid objects which belong to tubes in the Auslander-Reiten quiver of a cluster category of type $\widetilde A_{2,2}$. As usual, for any $m \geq 1$, we denote by $z_m$ the unique loop in $C_{2,2}$ going $m$ times around the annulus. Then it follows from Theorem \ref{theorem:BQatomic} that the atomic basis of $\mathcal A$ is:
	$$\mathcal B = \mathcal M \sqcup \ens{x_{z_m}x_{\alpha_i}^ax_{\beta_j}^b \ | \ m \geq 1, \, a,b \geq 0, \, 1 \leq i,j \leq 2}.$$

\section*{Acknowledgements}
 	This paper was written while the first author was a CRM-ISM postdoctoral fellow at the Universit\'e de Sherbrooke under the supervision of Ibrahim Assem, Thomas Br\"ustle and Virginie Charette and was also partially funded by the Tomlinson's Scholarship of Bishop's University. The second author was partially supported by an NSERC Discovery Grant. He worked on the paper during stays at the Fields Institute, the Hausdorff Centre, and Bielefeld University, whom he would like to thank for excellent working conditions. Both authors would like to thank Giovanni Cerulli Irelli for helpful discussions. 

\newcommand{\etalchar}[1]{$^{#1}$}

\end{document}